\documentclass[11pt]{article}

\usepackage[T1]{fontenc}
\usepackage{lmodern}
\usepackage{microtype}
\usepackage[margin=1in]{geometry}
\usepackage{amsmath,amssymb,amsthm,mathtools}
\usepackage[hidelinks]{hyperref}
\usepackage{bm}

\numberwithin{equation}{section}
\allowdisplaybreaks

\newtheorem{theorem}{Theorem}[section]
\newtheorem{proposition}[theorem]{Proposition}
\newtheorem{lemma}[theorem]{Lemma}
\newtheorem{corollary}[theorem]{Corollary}
\newtheorem{conjecture}[theorem]{Conjecture}
\theoremstyle{remark}
\newtheorem{remark}[theorem]{Remark}

\DeclareMathOperator{\Var}{Var}
\DeclareMathOperator{\Gap}{Gap}
\DeclareMathOperator{\supp}{supp}
\newcommand{\E}{\mathbb{E}}
\newcommand{\Pp}{\mathbb{P}}
\newcommand{\1}{\mathbf{1}}
\newcommand{\cE}{\mathcal{E}}
\newcommand{\Om}{\Omega}
\newcommand{\mun}{\mu_n}
\newcommand{\tmun}{\widetilde{\mu}_n}
\newcommand{\tPn}{\widetilde{P}_n}
\newcommand{\ip}[2]{\langle #1,#2\rangle}

\newcommand{\N}{\mathbb N}
\newcommand{\Pbb}{\mathbb P}
\newcommand{\tP}{\widetilde P}

\title{Conditionally Resampled Sliding-Window Count Kernels: Spectral-Gap Bounds and Poincar\'e Inequalities}
\author{
Yanjin Xiang\textsuperscript{a,$\dagger$},
Yuchen Xin\textsuperscript{a,$\dagger$},
Zhihua Zhang\textsuperscript{a,$\dagger$,*}\\[1ex]
\textsuperscript{a}
School of Mathematical Sciences, Peking University,\\ 5 Yiheyuan Road, Haidian District, Beijing 100871, China
}
\date{}

\begin{document}
\maketitle

\begingroup
\renewcommand{\thefootnote}{\fnsymbol{footnote}}
\footnotetext[2]{The authors are listed in alphabetical order.\\
\texttt{2401110086@stu.pku.edu.cn} (Y. Xiang); \\ \texttt{2301110087@pku.edu.cn} (Y. Xin); \\ \texttt{zhzhang@math.pku.edu.cn} (Z. Zhang).
}
\footnotetext[1]{Corresponding author.}
\endgroup

\begin{abstract}%
We study the conditionally resampled sliding-window count kernel associated with the empirical counts of length-$n$ windows from a stationary finite-state reversible Markov chain. Although the resulting count process is generally not Markov, its stationary one-step conditional law defines a genuine Markov kernel. For every fixed strictly positive reversible kernel \(P\) on a finite state space, we present a Poincaré inequality for the induced
count kernel $\tP_n$ of length $n$. In other words, we derive the lower bound of the spectral gap $\Gap(\tP_n)$ of $\tP_n$ as 
\[
\Gap(\tP_n)\ge \frac{c(P)}{n},
\]
where \(c(P)>0\) depends only on \(P\).  The proof combines a martingale oscillation inequality for the stationary path law with a direct comparison of coordinate oscillations to the Dirichlet form of the count kernel. A linear statistic of the count vector gives the matching \(O(1/n)\) upper bound, so for every fixed strictly positive reversible \(P\) one has \(\Gap(\tP_n)=\Theta_P(1/n)\). The resulting count-space Poincaré inequality yields a local-to-global variance bound for finite-window count statistics and, together with a general matrix-concentration principle, operator-norm concentration for matrix-valued empirical averages.

\end{abstract}



\newpage
\tableofcontents
\newpage

\section{Introduction}

Reversible Markov chains are fundamental stochastic processes that are widely used in statistics, machine learning, data science, and theoretical computer science. The convergence rate of Markov chains is important in related problems. Spectral gap is a key quantity that measures the rate at which the chain converges to its stationary distribution~\cite{levin2026markov,roch2024Book}. 
In this work, we  introduce a reversible Markov kernel on the empirical count space associated with length-$n$ stationary paths. We would study the spectral gap of this induced Markov chain. For this purpose, we develop Poincar\'e inequalities with respect to this induced Markov chain.
The meaning of the problem is that the spectral gap of the induced Markov chain may reflect some properties of the length-$n$ path walk of a Markov chain.

Given a finite state space $\Omega$, let  \(P=(p(x,y))_{x,y\in\Om}\) be a reversible Markov chain on $\Omega$.  For a length-$n$ stationary path
\[
(X_1,\ldots,X_n),
\]
we may project the path to its empirical count vector $N= (N_a)_{a\in\Omega}$. We  shift the path window from
\[
(X_1,\ldots,X_n)
\quad\text{to}\quad
(X_2,\ldots,X_{n+1})
\]
and then project them to counts, giving rise to a stationary process
\[
Y_t :=N(X_{t+1},\ldots,X_{t+n})
\]
on the finite simplex of count vectors. This projected process is generally not Markov, because the count vector does not retain enough information about the order of the hidden path. Functions of Markov chains are not Markov in general; classical conditions ensuring Markovianity are formulated in terms of lumpability or Markov functions; see \cite{KemenySnell,rogers1981markov,burke1958markovian}. Throughout this paper, we study instead the stationary one-step count kernel
\[
\tP_n(\eta,\xi) :=\Pbb(Y_1=\xi\mid Y_0=\eta),
\]
on the set of attainable count vectors, with stationary distribution given by the law of $Y_0$.

The question is to estimate the spectral gap of this one-step count kernel with respect to the window length \(n\).  The problem is nontrivial because the projection loses the order information in the path.  When \(|\Om|=2\), the count space is one-dimensional and the induced kernel is a birth-death chain.  When \(|\Om|\ge3\), the count space is a higher-dimensional simplex and the birth-death structure disappears.

The main contribution of this work is to prove a \(P\)-dependent lower bound for an arbitrary finite state space $\Omega$, under the clean positivity assumption
\[
p_*:=\min_{x,y\in\Om}p(x,y)>0.
\]
The positivity assumption is not needed to define the count kernel or to establish its reversibility and irreducibility on the support of its stationary law. It is used only in the proof of the quantitative spectral-gap lower bound, where it provides the uniform contraction and path-comparison estimates needed in the argument.


\subsection*{Background, main statement, and proof strategy}

The main result gives an explicit, although non-optimal, positive constant.  Let \(m=|\Om|\), write \(\pi\) for the stationary distribution of \(P\), and 
let $\mu_n(u_1, \ldots, u_n)=\pi(u_1) \prod_{i=1}^{n-1}p(u_i, u_{i+1})$ for $(u_1, \ldots, u_n) \in \Omega^{n}$. The Dobrushin contraction coefficient of the transition kernel
is defined as
\[
\delta(P)=
\sup_{x,y\in\Om}
\frac12\sum_{z\in\Om}|p(x,z)-p(y,z)|.
\]
This is the conventional Dobrushin, or total-variation contraction, coefficient of the transition kernel; see Dobrushin~\cite{Dobrushin1970} and the coupling formulation in Levin--Peres--Wilmer~\cite[Chapter~4]{levin2026markov}.  Equivalently, \(\delta(P)\) is the worst one-step total-variation distance between two rows of \(P\), and it controls the disagreement probability in the standard maximal-coupling construction used below.
Positivity $p_*>0$ implies \(\delta(P)<1\).  With
the notation, we derive the following bound.

\begin{theorem}[Spectral-gap lower bound for positive finite chains]
\label{thm:main}
Let \(P\) be a reversible Markov kernel on a finite state space \(\Om\) with \(m=|\Om|\ge2\), and assume \(p_*=\min_{x,y}p(x,y)>0\).  Then, for every \(n\ge2\), we have
\[
\Gap(\tP_n)
\ge
\frac{p_*^4(1-\sqrt{\delta(P)})^2}{4}\,\frac1n .
\]
\end{theorem}
Together with the linear-statistic upper bound in Section~\ref{sec:upper-bounds}, Theorem~\ref{thm:main} determines the $n$-dependence of the count-space spectral gap for every fixed strictly positive reversible kernel $P$:
\[
\Gap(\tP_n)=\Theta_P(n^{-1}).
\]
The explicit lower-bound constant is not claimed to have optimal dependence on $P$.


\paragraph{Proof strategy.}
Our proof has two main steps. We first regard a function of the count
vector as a function of the underlying ordered path. By revealing the
path from left to right and coupling two possible continuations, we
bound its variance under the stationary path law by the sum of the
squared changes caused by modifying one coordinate at a time. The
Dobrushin coefficient controls how the effect of such a modification
propagates along the path.

We then compare each one-coordinate change with the endpoint update
that defines the Dirichlet form of $\widetilde P_n$. A cyclic rotation
moves the chosen coordinate to the left endpoint, so that the
corresponding replacement can be compared directly with the endpoint
update appearing in the Dirichlet form. Summing over the $n$ possible
coordinates produces the factor $n$ in the Poincaré inequality.
Strict positivity provides the uniform bounds needed in these
comparisons.


The variance estimate for the stationary path law is related to
Dobrushin-type covariance and Poincaré estimates for Gibbs measures;
see, for example, Föllmer~\cite{follmer1982covariance} and
Wu~\cite{Wu2006}. The comparison of coordinate changes with the
endpoint Dirichlet form is in the spirit of comparison methods for
reversible Markov chains~\cite{diaconis1993comparison}. In our setting, however,
these results are not applied directly: we prove the path-space
variance estimate and then compare the coordinate changes directly
with the endpoint Dirichlet form of $\widetilde P_n$.


Beyond the proof itself, the count-space viewpoint connects the result
with occupation measures and finite-window statistics. The count
vector records the occupation counts of a stationary length-$n$
window, and its normalization by $n$ is the empirical measure of that
window. Empirical measures of Markov chains are classical objects in
large-deviation theory; see, for example, Donsker and
Varadhan~\cite{donsker1975asymptotic} and
Ellis~\cite{ellis1988large}. These works mainly concern
distributional or asymptotic properties of the counts, whereas our
object is the Poincaré spectral gap of the stationary one-step
resampling kernel on the count space.

The spectral-gap bound also yields a local-to-global variance
inequality for finite-window count statistics. It controls the
variance of a statistic of a stationary length-$n$ count vector by
its mean squared change under one actual shift of the window. Under
strict positivity, Theorem~\ref{thm:main} shows that the
prefactor in this variance bound is of order $n$, up to a constant
depending on $P$. From this functional-inequality viewpoint, the
result is analogous in spirit to estimates for Glauber, Kawasaki,
and conservative particle dynamics
\cite{lu1993spectral,quastel1992diffusion}. The dynamics considered
here is nevertheless different: the hidden path ordering is averaged
out, and the count update is inherited from a sliding-window endpoint
replacement rather than from a local spatial update.

Section~\ref{sec:matrix-concentration} applies the count-space Poincaré inequality to
matrix-valued empirical averages. Combining this scalar inequality
with the nonlinear matrix-concentration principle of Huang and
Tropp~\cite{huang2021poincare} yields operator-norm concentration and
an expectation bound on the $n^{-1/2}$ scale, up to logarithmic
dependence on the matrix dimension. Related matrix Poincaré methods
were developed by Aoun, Banna, and Youssef
\cite{aoun2020matrix}. This route is complementary to direct
concentration inequalities for matrix-valued sums along Markov
trajectories
\cite{qiu2020matrix,neeman2024concentration}.

\subsection*{The Paper Organization}

The remainder of this paper is organized as follows. Section~\ref{sec:preliminaries} gives preliminaries that are used in the paper. In Section~\ref{sec:count-kernel}, we define the resampled count kernel, establish its basic properties, and state the main Poincaré inequality. Section~\ref{sec:proof-lower-bound} gives the proof details of our main result on the lower bound (Theorem~\ref{thm:main-gap-lower-bound}).  Section~\ref{sec:upper-bounds}
gives a matching-order upper bound of interest. We discuss a special i.i.d.\ case in which the stationary chain is a multivariate Ehrenfest chain in Section~\ref{sec:iid-ehrenfest}, and consider a two-state case in Section~\ref{sec:two-state}.
Section~\ref{sec:matrix-concentration} illustrates applications of the count-space Poincaré inequality in matrix-valued empirical averages.
Further discussions are proposed in Section~\ref{sec:comparison-problem}.

\section{Notation and spectral-gap convention}
\label{sec:preliminaries}

All state spaces in the paper are finite.  If $\rho$ is a probability measure on a finite set $S$ and $f,g:S\to\mathbb R$, write
\[
\rho(f):=\sum_{x\in S}\rho(x)f(x),
\qquad
\ip{f}{g}_{\rho}:=\sum_{x\in S}\rho(x)f(x)g(x),
\]
and
\[
\Var_{\rho}(f)
:=\sum_{x\in S}\rho(x)\bigl(f(x)-\rho(f)\bigr)^2.
\]

Let $K$ be a Markov kernel that is reversible with respect to $\rho$.  Its Dirichlet form is
\begin{align}
\cE_K(f,g)
&:=\ip{f}{(I-K)g}_{\rho} \notag\\
&=\frac12\sum_{x,y\in S}\rho(x)K(x,y)
\bigl(f(y)-f(x)\bigr)\bigl(g(y)-g(x)\bigr).
\label{eq:prelim-dirichlet-form}
\end{align}
We use the \emph{right}, or Poincar\'e, spectral gap
\begin{equation}
\Gap(K)
:=
\inf_{\Var_{\rho}(f)>0}
\frac{\cE_K(f,f)}{\Var_{\rho}(f)}.
\label{eq:prelim-gap}
\end{equation}
For an irreducible reversible kernel, this equals $1-\lambda_2(K)$, where the eigenvalues are ordered as
\[
1=\lambda_1(K)>\lambda_2(K)\ge\cdots\ge -1.
\]
Throughout the paper, $\Gap(K)$ never denotes the absolute spectral gap.

If $(Z_0,Z_1)$ is a stationary one-step pair for $K$, then
\begin{equation}
\cE_K(f,f)
=
\frac12\E\bigl[(f(Z_1)-f(Z_0))^2\bigr].
\label{eq:prelim-stationary-pair}
\end{equation}
We shall also use the elementary path-reversal consequence of reversibility: if $(Z_t)_{t\ge 0}$ is stationary and reversible, then for every $r\ge 1$,
\begin{equation}
(Z_0,Z_1,\ldots,Z_r)
\overset{d}{=}
(Z_r,Z_{r-1},\ldots,Z_0).
\label{eq:prelim-path-reversal}
\end{equation}

\section{The conditionally resampled sliding-window count kernel}
\label{sec:count-kernel}

This section gives a self-contained definition of the count-space kernel
studied in the paper.  The construction has three ingredients: a stationary
law on ordered path segments, projection to empirical counts, and conditional
resampling of the hidden ordering before each sliding-window update.  We also
distinguish the resulting Markov chain from the original sliding-window count
process, which is generally not Markov.

\subsection{Stationary path law and empirical counts}
\label{subsec:path-law-counts}

Let $\Om$ be a finite state space with $m:=|\Om|\ge 2$, and let $P=(p(x,y))_{x,y\in\Om}$ be an irreducible Markov kernel that is reversible with respect to its
stationary distribution $\pi$. Let $X_1,X_2,\ldots$ be a stationary Markov chain with transition
kernel $P$.

Fix a window length $n\ge 2$.  The length-$n$ stationary path law on
$\Om^n$ is
\begin{equation}
\mun(u_1,\ldots,u_n)
:=
\Pp\bigl((X_1,\ldots,X_n)=(u_1,\ldots,u_n)\bigr)
=
\pi(u_1)\prod_{i=1}^{n-1}p(u_i,u_{i+1}).
\label{eq:stationary-path-law}
\end{equation}
Thus $\mun$ is a probability measure on ordered path segments.  In
particular, paths that contain the same states with different orderings need
not have the same $\mun$-probability.

For $u=(u_1,\ldots,u_n)\in\Om^n$, define its count vector by
\begin{equation}
N(u):=(N_x(u))_{x\in\Om},
\qquad
N_x(u):=\sum_{i=1}^n \1_{\{u_i=x\}},
\label{eq:empirical-count-map}
\end{equation}
and let
\begin{equation}
\Om_n^{\#}
:=
\left\{
\eta\in\mathbb N^{\Om}:
\sum_{x\in\Om}\eta_x=n
\right\}
\label{eq:count-simplex}
\end{equation}
denote the corresponding count space.  The map $N:\Om^n\to\Om_n^{\#}$ records how many
times each state appears and discards the order of appearance. We write \(u\sim\eta\) if \(N(u)=\eta\).

The push-forward of $\mun$ under $N$ is denoted by $\tmun$:
\begin{equation}
\tmun(\eta)
:=
\sum_{\substack{u\in\Om^n\\N(u)=\eta}}
\mun(u),
\qquad \eta\in\Om_n^{\#}.
\label{eq:stationary-count-law}
\end{equation}
If $P$ has zero transition probabilities, some count vectors may have zero
$\tmun$-mass.  We therefore define the effective count state space by
\begin{equation}
\Om_{n,+}^{\#}
:=
\supp(\tmun)
=
\{\eta\in\Om_n^{\#}:\tmun(\eta)>0\}.
\label{eq:count-support}
\end{equation}
No notational identification between $\Om_{n,+}^{\#}$ and the full simplex
$\Om_n^{\#}$ will be made until strict positivity of $P$ is imposed.

\subsection{Conditional resampling and the count transition}
\label{subsec:conditional-resampling-kernel}

Fix $\eta\in\Om_{n,+}^{\#}$.  A single transition is defined as follows.
\begin{enumerate}
\item Sample an ordered path $U=(U_1,\ldots,U_n)$ from the conditional
stationary path law
\begin{equation}
\Pp(U=u\mid N(U)=\eta)
=
\frac{\mun(u)}{\tmun(\eta)}\,
\1_{\{N(u)=\eta\}},
\qquad u\in\Om^n.
\label{eq:conditional-path-law}
\end{equation}

\item Conditional on $U$, sample a new state $Z$ according to
$P(U_n,\cdot)$.

\item Replace the count of the left endpoint by the count of the appended
state:
\begin{equation}
\eta' = \eta-e_{U_1}+e_Z,
\label{eq:count-update-rule}
\end{equation}
\end{enumerate}
where $e_x$ denotes the unit vector at coordinate $x$. 

The corresponding transition kernel on $\Om_{n,+}^{\#}$ is
\begin{equation}
\tPn(\eta,\xi)
=
\frac{1}{\tmun(\eta)}
\sum_{\substack{u\in\Om^n\\N(u)=\eta}}
\mun(u)
\sum_{z\in\Om}p(u_n,z)
\1_{\{\eta-e_{u_1}+e_z=\xi\}},
\qquad
\eta,\xi\in\Om_{n,+}^{\#}.
\label{eq:resampled-count-kernel}
\end{equation}
Whenever $\mun(u)p(u_n,z)>0$, the shifted path
$(u_2,\ldots,u_n,z)$ also has positive $\mun$-mass.  Hence the count vector
in \eqref{eq:count-update-rule} belongs to $\Om_{n,+}^{\#}$, and
\eqref{eq:resampled-count-kernel} indeed defines a Markov kernel on that
space.

We call $\tPn$ the \emph{conditionally resampled sliding-window count
kernel}, or simply the \emph{resampled count kernel}.  The word
``resampled'' is essential: conditional on the current count vector, the
hidden ordering is drawn from the generally nonuniform law in
\eqref{eq:conditional-path-law}.  Thus the construction is not, in general,
a uniform random reshuffling of the multiset represented by $\eta$.

To relate this kernel to the underlying sliding-window process, define the
ordered window and its count projection by
\begin{equation}
W_t:=(X_{t+1},\ldots,X_{t+n}),
\qquad
Y_t:=N(W_t),
\qquad t\ge 0.
\label{eq:sliding-window-processes}
\end{equation}
The process $(W_t)$ is Markov: it removes the first coordinate, appends a
state sampled from the transition row at the last coordinate, and keeps the
intervening coordinates in their existing order.  Its count projection
$(Y_t)$ is generally not Markov, because $N(W_t)$ does not determine the
ordered window and, in particular, does not determine its two endpoints.

Conditioning first on $W_0=(X_1,\ldots,X_n)$ gives
\begin{equation}
\tPn(\eta,\xi)
=
\Pp(Y_1=\xi\mid Y_0=\eta),
\qquad
\eta,\xi\in\Om_{n,+}^{\#}.
\label{eq:one-step-identification}
\end{equation}
Thus $\tPn$ is also the stationary one-step count kernel associated with
$(Y_t)$.

The one-step identification \eqref{eq:one-step-identification} does not
generally extend to multiple steps. Each transition of the $\tPn$-chain
begins by resampling an ordered path from the conditional law in
\eqref{eq:conditional-path-law}.  In the original sliding-window process,
by contrast, the ordered window produced by one shift is used directly in
the next shift. Consequently, for $k\ge 2$, $\tPn^k(\eta,\cdot)$ need not coincide with $\Pp(Y_k\in\cdot\mid Y_0=\eta)$.

\subsection{Reversibility, irreducibility, and the endpoint Dirichlet form}
\label{subsec:count-kernel-structure}

\begin{proposition}[Basic structure of the resampled count kernel]
\label{prop:count-kernel-structure}
The kernel $\tPn$ is reversible with respect to $\tmun$ and irreducible on
$\Om_{n,+}^{\#}$.  More precisely,
\begin{equation}
\tmun(\eta)\tPn(\eta,\xi)
=
\tmun(\xi)\tPn(\xi,\eta),
\qquad
\eta,\xi\in\Om_{n,+}^{\#}.
\label{eq:count-kernel-detailed-balance}
\end{equation}
Consequently, $\tmun$ is the unique stationary distribution of $\tPn$ on
$\Om_{n,+}^{\#}$.
\end{proposition}

\begin{proof}
By \eqref{eq:one-step-identification} and the fact that $Y_0$ has law
$\tmun$,
\begin{equation}
\tmun(\eta)\tPn(\eta,\xi)
=
\Pp(Y_0=\eta,Y_1=\xi).
\label{eq:joint-law-count-pair}
\end{equation}
Reversibility of $P$ implies that
\[
(X_1,\ldots,X_{n+1})
\overset{d}{=}
(X_{n+1},\ldots,X_1);
\]
see \eqref{eq:prelim-path-reversal}.  The forward path determines the pair
$(Y_0,Y_1)$, whereas the reversed path determines $(Y_1,Y_0)$, since
reversing the order inside a window does not change its count vector.
Therefore the joint law of $(Y_0,Y_1)$ is symmetric, and
\eqref{eq:count-kernel-detailed-balance} follows from
\eqref{eq:joint-law-count-pair}.

It remains to prove irreducibility.  Let
$\eta,\xi\in\Om_{n,+}^{\#}$.  Choose $u,v\in\Om^n$ such that
\[
N(u)=\eta,
\qquad
N(v)=\xi,
\qquad
\mun(u)>0,
\qquad
\mun(v)>0.
\]
Since $P$ is irreducible, there are states
\[
u_n=z_0,z_1,\ldots,z_r=v_1
\]
such that $p(z_{j-1},z_j)>0$ for $1\le j\le r$.  Starting from the ordered
window $u$, append successively
\[
z_1,\ldots,z_r,v_2,\ldots,v_n.
\]
This produces a finite sequence of admissible ordered windows
\[
w^{(0)}=u,w^{(1)},\ldots,w^{(L)}=v,
\]
where each shift appends a state with positive transition probability.  If
$a_j$ is the state appended in the transition from $w^{(j)}$ to
$w^{(j+1)}$, then \eqref{eq:resampled-count-kernel} gives
\begin{equation}
\tPn\bigl(N(w^{(j)}),N(w^{(j+1)})\bigr)
\ge
\frac{\mun(w^{(j)})}{\tmun(N(w^{(j)}))}
\,p\bigl(w^{(j)}_n,a_j\bigr)
>0.
\label{eq:positive-projected-transition}
\end{equation}
The projected sequence therefore connects $\eta$ to $\xi$ through
positive-probability transitions of $\tPn$.  This proves irreducibility.
Detailed balance gives stationarity of $\tmun$, and irreducibility gives
uniqueness.
\end{proof}

The next identity is the form of the Dirichlet energy used in the
spectral-gap argument.

\begin{proposition}[Endpoint representation of the Dirichlet form]
\label{prop:endpoint-dirichlet-form}
For every $F:\Om_{n,+}^{\#}\to\mathbb R$,
\begin{equation}
\begin{aligned}
\cE_{\tPn}(F,F)
={}
\frac12
\sum_{\substack{u\in\Om^n\\\mun(u)>0}}
\mun(u)
\sum_{\substack{z\in\Om\\p(u_n,z)>0}}
p(u_n,z)
\left[
F\bigl(N(u)-e_{u_1}+e_z\bigr)
-F\bigl(N(u)\bigr)
\right]^2.
\end{aligned}
\label{eq:endpoint-dirichlet-form}
\end{equation}
\end{proposition}

\begin{proof}
Whenever $\mun(u)p(u_n,z)>0$, the shifted path
$(u_2,\ldots,u_n,z)$ has positive $\mun$-mass.  Hence
\[
N(u)-e_{u_1}+e_z\in\Om_{n,+}^{\#},
\]
so every nonzero term below is well defined.  Substituting
\eqref{eq:resampled-count-kernel} into the definition of the Dirichlet
form gives
\begin{align*}
\cE_{\tPn}(F,F)
&=
\frac12
\sum_{\eta\in\Om_{n,+}^{\#}}
\tmun(\eta)
\sum_{\xi\in\Om_{n,+}^{\#}}
\tPn(\eta,\xi)
\bigl(F(\xi)-F(\eta)\bigr)^2
\\
&=
\frac12
\sum_{\eta\in\Om_{n,+}^{\#}}
\sum_{\substack{u\in\Om^n\\N(u)=\eta\\\mun(u)>0}}
\mun(u)
\sum_{\substack{z\in\Om\\p(u_n,z)>0}}
p(u_n,z)
\sum_{\xi\in\Om_{n,+}^{\#}}
\1_{\{\eta-e_{u_1}+e_z=\xi\}}
\bigl(F(\xi)-F(\eta)\bigr)^2
\\
&=
\frac12
\sum_{\substack{u\in\Om^n\\\mun(u)>0}}
\mun(u)
\sum_{\substack{z\in\Om\\p(u_n,z)>0}}
p(u_n,z)
\left[
F\bigl(N(u)-e_{u_1}+e_z\bigr)
-
F\bigl(N(u)\bigr)
\right]^2.
\end{align*}
This is \eqref{eq:endpoint-dirichlet-form}.

\end{proof}

Formula \eqref{eq:endpoint-dirichlet-form} shows that the count-space
Dirichlet form is generated by an endpoint replacement: the state at the
left endpoint of the hidden path is removed, and the appended state is
sampled from the transition row at the right endpoint.  The interior of the
path affects the energy through the stationary path weights and through the
conditional distribution of the hidden ordering.

\subsection{The Poincar\'e inequality and the main lower bound}
\label{subsec:main-lower-bound}

Having constructed the resampled count kernel $\tPn$, we now turn to
the main problem of the paper: understanding its Poincar\'e spectral
gap as the window length $n$ varies.  More precisely, we seek to
determine the dependence of $\Gap(\tPn)$ on $n$ for a fixed underlying
kernel $P$, while recording how the constants in the resulting bounds
depend on $P$.  Under strict positivity, the theorem below gives a
lower bound of order $n^{-1}$; a later linear-statistic argument gives
a matching upper bound of the same order.

For the quantitative lower bound,
we impose the stronger assumption
\begin{equation}
p_*
:=
\min_{x,y\in\Om}p(x,y)
>0.
\label{eq:strict-positivity}
\end{equation}
Under \eqref{eq:strict-positivity}, every ordered path has positive
$\mun$-mass and every count vector can be realized. Therefore, it is direct that
\begin{equation}
\Om_{n,+}^{\#}=\Om_n^{\#}.
\label{eq:full-count-support}
\end{equation}

Define the Dobrushin contraction coefficient of $P$ by
\begin{equation}
\delta(P)
:=
\max_{x,y\in\Om}
\frac12\sum_{z\in\Om}|p(x,z)-p(y,z)|.
\label{eq:dobrushin-coefficient-P}
\end{equation}
Since
\[
\frac12\sum_{z\in\Om}|p(x,z)-p(y,z)|
=
1-\sum_{z\in\Om}\min\{p(x,z),p(y,z)\},
\]
strict positivity gives
\begin{equation}
\delta(P)
\le
1-mp_*
<1.
\label{eq:positive-dobrushin-bound}
\end{equation}

\begin{theorem}[Poincar\'e inequality]
\label{thm:main-gap-lower-bound}
Let $P$ be a reversible Markov kernel on the finite state space $\Om$, and
assume \eqref{eq:strict-positivity}.  Then, for every $n\ge 2$ and every
function $F:\Om_n^{\#}\to\mathbb R$,
\begin{equation}
\Var_{\tmun}(F)
\le
\frac{4n}
{p_*^4\bigl(1-\sqrt{\delta(P)}\bigr)^2}
\,\cE_{\tPn}(F,F).
\label{eq:main-poincare-bound}
\end{equation}
Consequently,  we have the spectral-gap lower bound as
\begin{equation}
\Gap(\tPn)
\ge
\frac{p_*^4\bigl(1-\sqrt{\delta(P)}\bigr)^2}{4n}.
\label{eq:main-gap-bound}
\end{equation}
\end{theorem}

The proof of Theorem~\ref{thm:main-gap-lower-bound} is given in the next
section.  The displayed constant depends on the fixed kernel $P$ but not on
the window length $n$; no claim of optimality is made for this constant.  A
later linear-statistic argument gives the matching upper bound
$\Gap(\tPn)\le C(P)/n$, and hence identifies $n^{-1}$ as the correct order
for every fixed nontrivial strictly positive reversible kernel.

\subsection{A local-to-global inequality for finite-window count statistics}

The spectral gap of the resampled count kernel provides a
local-to-global inequality for statistics of a stationary window.
Indeed, by the Poincar\'e inequality, the one-step identification
in~\eqref{eq:one-step-identification}, and the endpoint
representation in Proposition~\ref{prop:endpoint-dirichlet-form},
every function
\[
F:\Omega^{\#}_{n,+}\longrightarrow\mathbb R
\]
satisfies
\begin{equation}
\operatorname{Var}_{\widetilde\mu_n}(F)
\le
\frac{1}{2\operatorname{Gap}(\widetilde P_n)}
\mathbb E\!\left[(F(Y_1)-F(Y_0))^2\right].
\label{eq:adjacent-window-poincare}
\end{equation}
Equivalently, the increment on the right-hand side is
\[
F\bigl(N(X_2,\ldots,X_{n+1})\bigr)
-
F\bigl(N(X_1,\ldots,X_n)\bigr).
\]
Thus the global fluctuation of a count statistic in a stationary
length-$n$ window is controlled by its mean squared change under
one actual shift of that window.  Estimating
$\operatorname{Gap}(\widetilde P_n)$ may therefore be viewed as
quantifying the cost of converting adjacent-window stability into
global fluctuation control.

General concentration inequalities for functions of a Markov path
are often formulated in terms of coordinatewise sensitivities,
martingale differences, and coupling or mixing coefficients
\cite{kontorovich2008concentration,Paulin2015}.  For additive
functionals, spectral methods for the underlying kernel provide
variance and concentration bounds with constants governed by its
spectral gap \cite{Paulin2015}.  Local-energy Poincar\'e
inequalities associated with conditional single-site updates are
also classical in the Gibbs-measure setting \cite{Wu2006}.

The viewpoint in \eqref{eq:adjacent-window-poincare} is
complementary.  It is formulated after the ordered path has been
projected to its count vector, and its Dirichlet energy is exactly
the averaged squared change under an actual shift of the
stationary window.  This form may provide a useful finite-sample
starting point for nonlinear or nonsmooth count statistics when
their averaged adjacent-window increment can be estimated more
sharply than a worst-case coordinatewise sensitivity.  We do not
pursue such function-specific estimates here.

A comparison of the form
\[
\operatorname{Gap}(\widetilde P_n)
\ge
c\,\frac{\operatorname{Gap}(P)}{n}
\]
would strengthen \eqref{eq:adjacent-window-poincare} to
\[
\operatorname{Var}_{\widetilde\mu_n}(F)
\le
\frac{n}{2c\,\operatorname{Gap}(P)}
\mathbb E\!\left[(F(Y_1)-F(Y_0))^2\right].
\]
In particular, a mean squared adjacent-window increment of order
$n^{-2}$ would then yield the natural variance scale
$(n\operatorname{Gap}(P))^{-1}$.  We formulate the corresponding
gap-comparison conjecture after the upper bounds in
Section~\ref{sec:upper-bounds}.

Finally, this interpretation is strictly one-step.  It uses the
joint law of $(Y_0,Y_1)$, but does not identify the multi-step
dynamics of the $\widetilde P_n$-chain with the original count
process $(Y_t)$.

\section{Proof of the lower bound}
\label{sec:proof-lower-bound}

The proof lifts a function on the count space to the stationary path
space. A martingale coupling controls the resulting path-space variance
by coordinate oscillations, which are then compared with the endpoint
Dirichlet form of the resampled count kernel \(\widetilde P_n\). We work
throughout under the positivity assumption \(p_\ast>0\), so
\(\delta(P)<1\).

\subsection{Two estimates and completion of the proof}
\label{subsec:two-estimates}

For a function \(H:\Omega^n\to\mathbb{R}\), a path
\(u=(u_1,\ldots,u_n)\in\Omega^n\), and \(1\le i\le n\), define the
\emph{coordinate oscillation}
\begin{align}
\operatorname{osc}_i H(u)
:=
\max_{a,b\in\Omega}
\Bigl|
&H(u_1,\ldots,u_{i-1},a,u_{i+1},\ldots,u_n)
-
H(u_1,\ldots,u_{i-1},b,u_{i+1},\ldots,u_n)
\Bigr|.
\label{eq:coordinate-oscillation}
\end{align}
Thus \(\operatorname{osc}_i H(u)\) is the largest possible change in
\(H\) when only the \(i\)-th coordinate is altered. The sum of the
squared coordinate oscillations will serve as the intermediate quantity
between the variance on path space and the Dirichlet form on count
space.

The first estimate controls the variance of an arbitrary function on
the stationary path space by this intermediate quantity.

\begin{proposition}[Variance bound by coordinate oscillations]
\label{prop:path-variance}
For every \(H:\Omega^n\to\mathbb{R}\),
\begin{equation}
\operatorname{Var}_{\mu_n}(H)
\le
\frac{1}
{p_\ast\bigl(1-\sqrt{\delta(P)}\bigr)^2}
\sum_{i=1}^n
\mathbb{E}_{\mu_n}
\!\left[
  \bigl(\operatorname{osc}_i H\bigr)^2
\right].
\label{eq:path-variance-bound}
\end{equation}
\end{proposition}

Related martingale- and coupling-based estimates for general functions of
Markov paths appear in
\cite{kontorovich2008concentration,redig2009concentration}.
Proposition~\ref{prop:path-variance} is tailored to the present comparison
argument: it is formulated in terms of the averaged local coordinate
oscillations
\[
\mathbb{E}_{\mu_n}\bigl[(\operatorname{osc}_i H)^2\bigr]
\]
and gives an explicit constant in terms of $p_*$ and $\delta(P)$.
We therefore include a direct proof.

The second estimate applies to functions that depend on the path only
through its count vector. It compares the same intermediate quantity
with the endpoint Dirichlet form of the resampled count kernel.

\begin{proposition}[Coordinate oscillations controlled by the count-chain Dirichlet form]
\label{prop:coordinate-oscillation-dirichlet-bound}
For every \(F:\Omega_n^{\#}\to\mathbb{R}\), let
\[
H:=F\circ N.
\]
Then
\begin{equation}
\sum_{i=1}^n
\mathbb{E}_{\mu_n}
\!\left[
  \bigl(\operatorname{osc}_i H\bigr)^2
\right]
\le
\frac{4n}{p_\ast^3}\,
\mathcal{E}_{\widetilde{P}_n}(F,F).
\label{eq:coordinate-oscillation-dirichlet-bound}
\end{equation}
\end{proposition}

The two propositions isolate the two different mechanisms in the
argument. Proposition~\ref{prop:path-variance} is entirely a statement
about the stationary path law \(\mu_n\). Its proof reveals the path from
left to right and couples two possible conditional futures; the
contraction coefficient \(\delta(P)\) makes their disagreement
probabilities decay geometrically.

Proposition~\ref{prop:coordinate-oscillation-dirichlet-bound} is where the count map and
the particular update of \(\widetilde{P}_n\) enter. Replacing \(u_i\) by
a state \(z\) changes the count vector from \(N(u)\) to
\[
N(u)-e_{u_i}+e_z.
\]
A cyclic rotation that places \(u_i\) at the left endpoint turns this
replacement into the endpoint update appearing in the Dirichlet-form
representation from Section~\ref{sec:count-kernel}. In particular, the factor \(n\) in the
final bound enters only when the \(n\) possible replacement positions
are summed in Proposition~\ref{prop:coordinate-oscillation-dirichlet-bound}.

Before proving Propositions~\ref{prop:path-variance} and \ref{prop:coordinate-oscillation-dirichlet-bound} in Sections~\ref{subsec:path-variance} and \ref{subsec:coordinate-oscillations-to-dirichlet}, we show how they imply Theorem~\ref{thm:main-gap-lower-bound}.

\begin{proof}[Proof of Theorem~\ref{thm:main-gap-lower-bound}]
Let \(F:\Omega_n^{\#}\to\mathbb{R}\), and set
\[
H:=F\circ N.
\]
Since \(\widetilde{\mu}_n\) is the push-forward of \(\mu_n\) under
\(N\),
\begin{equation}
\operatorname{Var}_{\widetilde{\mu}_n}(F)
=
\operatorname{Var}_{\mu_n}(H).
\label{eq:variance-pushforward}
\end{equation}
Applying Propositions~\ref{prop:path-variance} and
\ref{prop:coordinate-oscillation-dirichlet-bound}, we obtain
\begin{align*}
\operatorname{Var}_{\widetilde{\mu}_n}(F)
&\le
\frac{1}
{p_\ast\bigl(1-\sqrt{\delta(P)}\bigr)^2}
\sum_{i=1}^n
\mathbb{E}_{\mu_n}
\!\left[
  \bigl(\operatorname{osc}_i H\bigr)^2
\right]
\\
&\le
\frac{4n}
{p_\ast^4\bigl(1-\sqrt{\delta(P)}\bigr)^2}
\mathcal{E}_{\widetilde{P}_n}(F,F).
\end{align*}
This proves the asserted Poincar\'e inequality. Taking the infimum in
the variational definition of
\(\operatorname{Gap}(\widetilde P_n)\) gives the claimed spectral-gap
bound.

\end{proof}

It remains to prove Propositions~\ref{prop:path-variance} and
\ref{prop:coordinate-oscillation-dirichlet-bound}. We begin with the variance estimate on
the stationary path space.

\subsection{A variance estimate for the stationary path law}
\label{subsec:path-variance}

We now prove Proposition~\ref{prop:path-variance}. The proof reveals the
path from left to right and uses a maximal coupling to compare two possible
conditional futures. We then pass from one coupled path to the other by
changing one coordinate at a time.

The intermediate paths arising from this coordinate-by-coordinate
comparison need not themselves have a conditional Markov path law. The
only additional point in the argument is therefore to compare their
probabilities with those of paths drawn from the original conditional law.
At the point where two pieces of the coupled trajectories are joined, at
most one transition factor is missing. The lower bound \(p_\ast\) controls
the cost of restoring this factor.

\begin{proof}[Proof of Proposition~\ref{prop:path-variance}]
Let
\[
U=(U_1,\ldots,U_n)
\]
have law \(\mu_n\), and set
\[
\alpha:=\sqrt{\delta(P)}\in[0,1).
\]
For \(0\le i\le n\), let
\[
\mathcal F_i:=\sigma(U_1,\ldots,U_i),
\]
where \(\mathcal F_0\) is the trivial \(\sigma\)-field, and define the
martingale differences
\[
\Delta_i
:=
\mathbb E_{\mu_n}\!\left[H(U)\mid\mathcal F_i\right]
-
\mathbb E_{\mu_n}\!\left[H(U)\mid\mathcal F_{i-1}\right],
\qquad 1\le i\le n.
\]
The orthogonality of martingale differences gives
\begin{equation}
\operatorname{Var}_{\mu_n}(H)
=
\sum_{i=1}^n
\mathbb E_{\mu_n}[\Delta_i^2].
\label{eq:doob-variance-decomposition}
\end{equation}

We shall prove that, for every \(1\le i\le n\),
\begin{equation}
\mathbb E_{\mu_n}[\Delta_i^2]
\le
\frac{1}{p_\ast(1-\alpha)}
\sum_{j=i}^n
\alpha^{j-i}
\mathbb E_{\mu_n}
\!\left[
  \bigl(\operatorname{osc}_j H\bigr)^2
\right].
\label{eq:martingale-increment-bound}
\end{equation}
The proposition will then follow by summing this estimate over \(i\).

\medskip
\noindent\textbf{Step 1. Conditioning on the revealed past.}

Fix \(i\), and fix a value
\[
w=(w_1,\ldots,w_{i-1})
\]
of \((U_1,\ldots,U_{i-1})\) having positive probability. When \(i=1\),
the past \(w\) is empty. Let \(\rho_w\) denote the conditional law of
\(U_i\) given this past, and define
\[
G_w(c)
:=
\mathbb E_{\mu_n}
\!\left[
  H(U)
  \,\middle|\,
  U_1=w_1,\ldots,U_{i-1}=w_{i-1},\ U_i=c
\right],
\qquad c\in\Omega.
\]
Strict positivity ensures that every conditional expectation above is
well defined.

If \(A\) and \(B\) are independent random variables with law \(\rho_w\),
then
\begin{equation}
\begin{aligned}
\mathbb E_{\mu_n}
\!\left[
  \Delta_i^2
  \,\middle|\,
  U_1=w_1,\ldots,U_{i-1}=w_{i-1}
\right]
=
\operatorname{Var}_{\rho_w}(G_w)
=
\frac12\,
\mathbb E\!\left[
  \bigl(G_w(A)-G_w(B)\bigr)^2
\right].
\end{aligned}
\label{eq:conditional-martingale-variance}
\end{equation}
It therefore remains to control \(G_w(a)-G_w(b)\) for fixed
\(a,b\in\Omega\).

\medskip
\noindent\textbf{Step 2. Coupling two conditional futures.}

To estimate \(G_w(a)-G_w(b)\), we place the two conditional
futures on a common probability space. The coupling is chosen so that
each marginal trajectory has the required conditional Markov path law and, once
the two trajectories meet, they remain together.

Fix \(a,b\in\Omega\). For \(c\in\{a,b\}\), the Markov property of
the stationary path law gives
\[
\begin{aligned}
&\mu_n
\!\left(
  U_{i+1}=x_{i+1},\ldots,U_n=x_n
  \,\middle|\,
  U_1=w_1,\ldots,U_{i-1}=w_{i-1},\ U_i=c
\right)
\\
&\qquad=
\prod_{r=i}^{n-1}p(x_r,x_{r+1}),
\qquad x_i:=c.
\end{aligned}
\]
Thus the conditional law entering \(G_w(c)\) is obtained by fixing
the past \(w\), placing the chain at \(c\) at time \(i\), and then
evolving with transition kernel \(P\).

We now construct two random paths
\[
\begin{aligned}
U^a
&=
(w_1,\ldots,w_{i-1},a,U_{i+1}^a,\ldots,U_n^a),
\\
U^b
&=
(w_1,\ldots,w_{i-1},b,U_{i+1}^b,\ldots,U_n^b),
\end{aligned}
\]
on the same probability space. Set $U_i^a=a,
U_i^b=b$.

We use the standard Markovian construction in which the two one-step
transition laws are maximally coupled at each time; see, for example,
\cite[Proposition~4.7 and Section~5.1]{levin2026markov}.
For every \(x,y\in\Omega\), choose a maximal coupling \(Q_{x,y}\)
of \(P(x,\cdot)\) and \(P(y,\cdot)\). Thus
\[
\sum_{y'\in\Omega}Q_{x,y}(x',y')
=
p(x,x'),
\qquad
\sum_{x'\in\Omega}Q_{x,y}(x',y')
=
p(y,y'),
\]
and
\[
Q_{x,y}
\bigl\{
  (x',y')\in\Omega^2:x'\ne y'
\bigr\}
=
\left\|P(x,\cdot)-P(y,\cdot)\right\|_{\mathrm{TV}}
\le
\delta(P).
\]

For \(t\ge i\), let
\[
\mathcal G_t
:=
\sigma
\bigl(
  U_i^a,U_i^b,\ldots,U_t^a,U_t^b
\bigr)
\]
be the history of the coupled trajectories up to time \(t\). We define
the future coordinates recursively by requiring that, for
\(t=i,\ldots,n-1\) and \(x',y'\in\Omega\),
\[
\mathbb P
\!\left(
  U_{t+1}^a=x',\ U_{t+1}^b=y'
  \,\middle|\,
  \mathcal G_t
\right)
=
Q_{U_t^a,U_t^b}(x',y').
\]
Taking the two marginals in this identity gives
\[
\mathbb P
\!\left(
  U_{t+1}^a=x'
  \,\middle|\,
  \mathcal G_t
\right)
=
p(U_t^a,x')
\]
and
\[
\mathbb P
\!\left(
  U_{t+1}^b=y'
  \,\middle|\,
  \mathcal G_t
\right)
=
p(U_t^b,y').
\]
Iterating these transition identities shows that, for
\(c\in\{a,b\}\),
\[
\mathbb P
\!\left(
  U_{i+1}^c=x_{i+1},\ldots,U_n^c=x_n
\right)
=
\prod_{r=i}^{n-1}p(x_r,x_{r+1}),
\qquad x_i:=c.
\]
Consequently,
\[
U^a
\sim
\mu_n
\!\left(
  \,\cdot\,
  \middle|\,
  U_1=w_1,\ldots,U_{i-1}=w_{i-1},\ U_i=a
\right),
\]
and \(U^b\) has the analogous conditional law with \(b\) in place of
\(a\). In particular,
\begin{equation}
G_w(a)=\mathbb E[H(U^a)],
\qquad
G_w(b)=\mathbb E[H(U^b)].
\label{eq:conditional-expectations-under-coupling}
\end{equation}

If \(U_t^a=U_t^b=x\), then the two one-step distributions being
coupled are both \(P(x,\cdot)\). Their total-variation distance is zero,
so maximality implies $U_{t+1}^a=U_{t+1}^b$ almost surely.
Hence, once the two trajectories meet, they remain equal at all later
times. More generally, the maximal-coupling property gives
\[
\begin{aligned}
\mathbb P
\!\left(
  U_{t+1}^a\ne U_{t+1}^b
  \,\middle|\,
  \mathcal G_t
\right)
=
\left\|
  P(U_t^a,\cdot)-P(U_t^b,\cdot)
\right\|_{\mathrm{TV}}
\le
\delta(P)\,
\mathbf 1_{\{U_t^a\ne U_t^b\}}.
\end{aligned}
\]
Taking expectations yields
\[
\mathbb P(U_{t+1}^a\ne U_{t+1}^b)
\le
\delta(P)\,
\mathbb P(U_t^a\ne U_t^b).
\]
Since
\(\mathbb P(U_i^a\ne U_i^b)\le1\), induction gives
\begin{equation}
\mathbb P(U_j^a\ne U_j^b)
\le
\delta(P)^{j-i},
\qquad j\ge i.
\label{eq:coupled-disagreement}
\end{equation}

All probabilities and expectations in the next two steps refer to this
coupling, unless a subscript is displayed explicitly.

\medskip
\noindent\textbf{Step 3. Changing one coordinate at a time.}

For \(i-1\le r\le n\), define the intermediate path
\[
V^{(r)}
:=
(w_1,\ldots,w_{i-1},
 U_i^b,\ldots,U_r^b,
 U_{r+1}^a,\ldots,U_n^a),
\]
with the convention that either displayed block is omitted when it is
empty. Thus
\[
V^{(i-1)}=U^a,
\qquad
V^{(n)}=U^b.
\]
Telescoping gives
\[
H(U^a)-H(U^b)
=
\sum_{j=i}^n
\left(
  H(V^{(j-1)})-H(V^{(j)})
\right).
\]
The two paths in the \(j\)-th summand differ only at coordinate \(j\).
Hence, defining
\[
O_j
:=
\bigl(\operatorname{osc}_jH\bigr)(V^{(j-1)}),
\]
we have
\[
\left|
  H(V^{(j-1)})-H(V^{(j)})
\right|
\le
\mathbf 1_{\{U_j^a\ne U_j^b\}}\,O_j.
\]

By \eqref{eq:conditional-expectations-under-coupling},
\[
G_w(a)-G_w(b)
=
\mathbb E\!\left[H(U^a)-H(U^b)\right].
\]
It follows from Cauchy--Schwarz and
\eqref{eq:coupled-disagreement} that
\begin{align*}
|G_w(a)-G_w(b)|
&\le
\sum_{j=i}^n
\mathbb E
\!\left[
  \mathbf 1_{\{U_j^a\ne U_j^b\}}\,O_j
\right]
\\
&\le
\sum_{j=i}^n
\mathbb P(U_j^a\ne U_j^b)^{1/2}
\bigl(\mathbb E[O_j^2]\bigr)^{1/2}
\\
&\le
\sum_{j=i}^n
\alpha^{j-i}
\bigl(\mathbb E[O_j^2]\bigr)^{1/2}.
\end{align*}
Applying Cauchy--Schwarz once more to the sum over \(j\), we obtain
\begin{align}
|G_w(a)-G_w(b)|^2
&\le
\left(
  \sum_{j=i}^n\alpha^{j-i}
\right)
\left(
  \sum_{j=i}^n
  \alpha^{j-i}\mathbb E[O_j^2]
\right)
\nonumber\\
&\le
\frac{1}{1-\alpha}
\sum_{j=i}^n
\alpha^{j-i}\mathbb E[O_j^2].
\label{eq:coupled-future-telescoping}
\end{align}

\medskip
\noindent\textbf{Step 4. Comparing the intermediate-path distribution
with the corresponding conditional law.}

For \(c\in\Omega\) and \(j\ge i\), write
\[
L_j(w,c)
:=
\mathbb E_{\mu_n}
\!\left[
  \bigl(\operatorname{osc}_jH(U)\bigr)^2
  \,\middle|\,
  U_1=w_1,\ldots,U_{i-1}=w_{i-1},\ U_i=c
\right].
\]

When \(j=i\), the intermediate path
\[
V^{(i-1)}=U^a
\]
has the original conditional path law started from \(a\). Therefore
\begin{equation}
\mathbb E[O_i^2]
=
L_i(w,a).
\label{eq:first-intermediate-path}
\end{equation}

Now suppose \(j>i\). The intermediate path \(V^{(j-1)}\) uses the
\(b\)-trajectory through coordinate \(j-1\) and the \(a\)-trajectory
from coordinate \(j\) onward. We compare its distribution with the
conditional path law given the same past \(w\) and the state \(b\) at
time \(i\).

Fix
\[
v
=
(w_1,\ldots,w_{i-1},b,v_{i+1},\ldots,v_n)
\in\Omega^n,
\]
and put \(v_i:=b\). Our goal is to prove the pointwise domination
\[
\mathbb P\bigl(V^{(j-1)}=v\bigr)
\le
\frac{1}{p_\ast}
\mu_n
\!\left(
  v
  \,\middle|\,
  U_1=w_1,\ldots,U_{i-1}=w_{i-1},\ U_i=b
\right).
\]

We first estimate the probability on the left-hand side. Let
\[
B_{j-1}(v)
:=
\{
  U_i^b=v_i,\ldots,U_{j-1}^b=v_{j-1}
\}.
\]
Because the \(U^b\)-trajectory evolves with transition kernel \(P\),
\begin{equation}
\mathbb P(B_{j-1}(v))
=
\prod_{r=i}^{j-2}p(v_r,v_{r+1}).
\label{eq:b-prefix-probability}
\end{equation}
In particular, this probability is positive under the strict
positivity assumption.

By the definition of \(V^{(j-1)}\),
\[
\{
  V^{(j-1)}=v
\}
=
B_{j-1}(v)
\cap
\{
  U_j^a=v_j,\ldots,U_n^a=v_n
\}.
\]
For \(j\le r\le n\), define
\[
C_r(v)
:=
B_{j-1}(v)
\cap
\{
  U_j^a=v_j,\ldots,U_r^a=v_r
\}.
\]
Then
\[
C_n(v)=\{V^{(j-1)}=v\},
\qquad
C_j(v)\subseteq B_{j-1}(v).
\]

For \(j\le r<n\), the event \(C_r(v)\) belongs to
\(\mathcal G_r\). Hence the conditional transition identity from
Step~2 gives
\begin{align*}
\mathbb P(C_{r+1}(v))
&=
\mathbb E
\!\left[
  \mathbf 1_{C_r(v)}
  \mathbb P
  \!\left(
    U_{r+1}^a=v_{r+1}
    \,\middle|\,
    \mathcal G_r
  \right)
\right]
\\
&=
\mathbb E
\!\left[
  \mathbf 1_{C_r(v)}
  p(U_r^a,v_{r+1})
\right]
\\
&=
p(v_r,v_{r+1})\,
\mathbb P(C_r(v)).
\end{align*}
Iterating this identity from \(r=j\) to \(r=n-1\), we obtain
\[
\mathbb P(V^{(j-1)}=v)
=
\mathbb P(C_j(v))
\prod_{r=j}^{n-1}p(v_r,v_{r+1}).
\]
Since \(C_j(v)\subseteq B_{j-1}(v)\),
\[
\mathbb P(C_j(v))
\le
\mathbb P(B_{j-1}(v))
=
\prod_{r=i}^{j-2}p(v_r,v_{r+1}).
\]
Consequently,
\begin{equation}
\mathbb P(V^{(j-1)}=v)
\le
\left(
  \prod_{r=i}^{j-2}p(v_r,v_{r+1})
\right)
\left(
  \prod_{r=j}^{n-1}p(v_r,v_{r+1})
\right).
\label{eq:intermediate-path-upper-bound}
\end{equation}

On the other hand, under the conditional path law given the same past
\(w\) and \(U_i=b\), the path \(v\) must also make the transition from
\(v_{j-1}\) to \(v_j\). Therefore
\begin{align}
&
\mu_n
\!\left(
  v
  \,\middle|\,
  U_1=w_1,\ldots,U_{i-1}=w_{i-1},\ U_i=b
\right)
\nonumber\\
&\qquad=
\left(
  \prod_{r=i}^{j-2}p(v_r,v_{r+1})
\right)
p(v_{j-1},v_j)
\left(
  \prod_{r=j}^{n-1}p(v_r,v_{r+1})
\right).
\label{eq:conditional-path-probability}
\end{align}
Empty products are understood to be one. Since
\[
p(v_{j-1},v_j)\ge p_\ast,
\]
equations \eqref{eq:intermediate-path-upper-bound} and
\eqref{eq:conditional-path-probability} imply
\begin{equation}
\mathbb P(V^{(j-1)}=v)
\le
\frac{1}{p_\ast}
\mu_n
\!\left(
  v
  \,\middle|\,
  U_1=w_1,\ldots,U_{i-1}=w_{i-1},\ U_i=b
\right).
\label{eq:intermediate-path-comparison}
\end{equation}

Integrating the nonnegative function
\(\bigl(\operatorname{osc}_jH\bigr)^2\) in
\eqref{eq:intermediate-path-comparison} gives
\begin{equation}
\mathbb E[O_j^2]
\le
\frac{1}{p_\ast}L_j(w,b),
\qquad j>i.
\label{eq:intermediate-oscillation-comparison}
\end{equation}
Combining \eqref{eq:first-intermediate-path} and
\eqref{eq:intermediate-oscillation-comparison}, and using
\(p_\ast\le1\), yields the uniform estimate
\begin{equation}
\mathbb E[O_j^2]
\le
\frac{1}{p_\ast}
\bigl(L_j(w,a)+L_j(w,b)\bigr),
\qquad j\ge i.
\label{eq:uniform-intermediate-oscillation-bound}
\end{equation}

Substituting \eqref{eq:uniform-intermediate-oscillation-bound} into
\eqref{eq:coupled-future-telescoping}, we obtain
\begin{equation}
|G_w(a)-G_w(b)|^2
\le
\frac{1}{p_\ast(1-\alpha)}
\sum_{j=i}^n
\alpha^{j-i}
\bigl(L_j(w,a)+L_j(w,b)\bigr).
\label{eq:conditional-future-difference}
\end{equation}

\medskip
\noindent\textbf{Step 5. Averaging the conditional bound and summing
the increment estimates.}

Average \eqref{eq:conditional-future-difference} over independent
\(A,B\sim\rho_w\). Using
\eqref{eq:conditional-martingale-variance}, we obtain
\begin{align*}
&
\mathbb E_{\mu_n}
\!\left[
  \Delta_i^2
  \,\middle|\,
  U_1=w_1,\ldots,U_{i-1}=w_{i-1}
\right]
\\
&\qquad\le
\frac{1}{2p_\ast(1-\alpha)}
\sum_{j=i}^n
\alpha^{j-i}
\mathbb E
\!\left[
  L_j(w,A)+L_j(w,B)
\right]
\\
&\qquad=
\frac{1}{p_\ast(1-\alpha)}
\sum_{j=i}^n
\alpha^{j-i}
\sum_{c\in\Omega}\rho_w(c)L_j(w,c).
\end{align*}
By the definition of \(L_j\) and the tower property,
\[
\sum_{c\in\Omega}\rho_w(c)L_j(w,c)
=
\mathbb E_{\mu_n}
\!\left[
  \bigl(\operatorname{osc}_jH(U)\bigr)^2
  \,\middle|\,
  U_1=w_1,\ldots,U_{i-1}=w_{i-1}
\right].
\]
Taking expectation over the revealed past proves
\eqref{eq:martingale-increment-bound}.

Finally, summing \eqref{eq:martingale-increment-bound} over \(i\) and
using \eqref{eq:doob-variance-decomposition}, we obtain
\begin{align*}
\operatorname{Var}_{\mu_n}(H)
&\le
\frac{1}{p_\ast(1-\alpha)}
\sum_{i=1}^n
\sum_{j=i}^n
\alpha^{j-i}
\mathbb E_{\mu_n}
\!\left[
  \bigl(\operatorname{osc}_jH\bigr)^2
\right]
\\
&=
\frac{1}{p_\ast(1-\alpha)}
\sum_{j=1}^n
\left(
  \sum_{i=1}^j\alpha^{j-i}
\right)
\mathbb E_{\mu_n}
\!\left[
  \bigl(\operatorname{osc}_jH\bigr)^2
\right]
\\
&\le
\frac{1}{p_\ast(1-\alpha)^2}
\sum_{j=1}^n
\mathbb E_{\mu_n}
\!\left[
  \bigl(\operatorname{osc}_jH\bigr)^2
\right].
\end{align*}
Since \(\alpha=\sqrt{\delta(P)}\), this is precisely
\eqref{eq:path-variance-bound}.
\end{proof}

\subsection{From coordinate oscillations to the count-chain Dirichlet form}
\label{subsec:coordinate-oscillations-to-dirichlet}

We now prove
Proposition~\ref{prop:coordinate-oscillation-dirichlet-bound}.
The endpoint representation
\eqref{eq:endpoint-dirichlet-form} expresses the Dirichlet form as an average of squared updates at the left endpoint, with the appended state sampled from the transition row at the right endpoint. For an arbitrary coordinate $i$, a cyclic rotation places $u_i$ at the left endpoint and $u_{i-1}$ at the right endpoint. The proof has two ingredients: a pointwise oscillation bound for the corresponding one-site replacement and a density comparison under cyclic rotation.

We first isolate the pointwise ingredient. The following elementary inequality on a finite set bounds the range of a function by its weighted squared deviations from its value at a fixed reference point.

\begin{lemma}[Oscillation bound from a reference state]
\label{lem:reference-state-oscillation-bound}
Let \(q\) be a probability measure on a finite set \(S\), and suppose
\[
    q_{\min}:=\min_{x\in S}q(x)>0.
\]
Then, for every \(h:S\to\mathbb{R}\) and every \(a_0\in S\),
\begin{equation}
    \left(
        \max_{x\in S}h(x)-\min_{x\in S}h(x)
    \right)^2
    \le
    \frac{2}{q_{\min}}
    \sum_{x\in S}
    q(x)\bigl(h(x)-h(a_0)\bigr)^2.
    \label{eq:reference-state-oscillation-bound}
\end{equation}
\end{lemma}

\begin{proof}
The claim is immediate if \(h\) is constant. Otherwise, choose
\(b,c\in S\) such that
\[
    h(b)=\max_{x\in S}h(x),
    \qquad
    h(c)=\min_{x\in S}h(x).
\]
Since \(b\ne c\),
\begin{align*}
    \sum_{x\in S}
    q(x)\bigl(h(x)-h(a_0)\bigr)^2
    &\ge
    q_{\min}
    \left[
        \bigl(h(b)-h(a_0)\bigr)^2
        +
        \bigl(h(c)-h(a_0)\bigr)^2
    \right]
    \\
    &\ge
    \frac{q_{\min}}{2}
    \bigl(h(b)-h(c)\bigr)^2,
\end{align*}
where the last inequality follows from
\[
    (x-y)^2\le 2x^2+2y^2.
\]
Rearranging proves
\eqref{eq:reference-state-oscillation-bound}.
\end{proof}

\begin{proof}[Proof of Proposition~\ref{prop:coordinate-oscillation-dirichlet-bound}]
Fix \(F:\Omega_n^\#\to\mathbb{R}\), and set
\[
    H:=F\circ N.
\]

For a path \(u=(u_1,\ldots,u_n)\in\Omega^n\), define the endpoint-update
cost
\begin{equation}
    D_1(u)
    :=
    \sum_{z\in\Omega}
    p(u_n,z)
    \left[
        F\bigl(N(u)-e_{u_1}+e_z\bigr)
        -
        F\bigl(N(u)\bigr)
    \right]^2.
    \label{eq:endpoint-update-cost}
\end{equation}
By the endpoint representation
\eqref{eq:endpoint-dirichlet-form},
\begin{equation}
    \mathcal{E}_{\widetilde P_n}(F,F)
    =
    \frac12\,
    \mathbb{E}_{\mu_n}[D_1].
    \label{eq:endpoint-cost-dirichlet-form}
\end{equation}

To obtain the corresponding quantity at an arbitrary coordinate,
use the cyclic convention \(u_0:=u_n\), and for \(1\le i\le n\) define
\begin{equation}
    \theta_i(u_1,\ldots,u_n)
    :=
    (u_i,u_{i+1},\ldots,u_n,u_1,\ldots,u_{i-1}).
    \label{eq:cyclic-rotation}
\end{equation}
The rotation preserves the count vector and places \(u_i\) and
\(u_{i-1}\) at the two endpoints:
\[
    N(\theta_i u)=N(u),
    \qquad
    (\theta_i u)_1=u_i,
    \qquad
    (\theta_i u)_n=u_{i-1}.
\]
We therefore define
\begin{align}
    D_i(u)
    &:=
    D_1(\theta_i u)
    \notag\\
    &=
    \sum_{z\in\Omega}
    p(u_{i-1},z)
    \left[
        F\bigl(N(u)-e_{u_i}+e_z\bigr)
        -
        F\bigl(N(u)\bigr)
    \right]^2.
    \label{eq:rotated-endpoint-cost}
\end{align}
Thus \(D_i\) is the endpoint-update quantity evaluated on the
cyclically rotated path \(\theta_i u\).

\medskip
\noindent
\textbf{Step 1: Pointwise control of a coordinate oscillation.}

We first prove that, for every \(u\in\Omega^n\) and
\(1\le i\le n\),
\begin{equation}
    \bigl(\operatorname{osc}_i H(u)\bigr)^2
    \le
    \frac{2}{p_*}\,D_i(u).
    \label{eq:pointwise-oscillation-cost-bound}
\end{equation}

Fix \(u\in\Omega^n\) and \(1\le i\le n\), and define
\[
    h(z)
    :=
    F\bigl(N(u)-e_{u_i}+e_z\bigr),
    \qquad z\in\Omega.
\]
Replacing the \(i\)-th coordinate of \(u\) by \(z\) changes its count
vector from \(N(u)\) to \(N(u)-e_{u_i}+e_z\). Therefore,
\begin{equation}
    \operatorname{osc}_i H(u)
    =
    \max_{z\in\Omega}h(z)
    -
    \min_{z\in\Omega}h(z).
    \label{eq:coordinate-oscillation-as-range}
\end{equation}
Moreover,
\[
    h(u_i)=F\bigl(N(u)\bigr).
\]

We apply
Lemma~\ref{lem:reference-state-oscillation-bound}
with
\[
    S=\Omega,
    \qquad
    q(z)=p(u_{i-1},z),
    \qquad
    a_0=u_i.
\]
Since
\[
    \min_{z\in\Omega}p(u_{i-1},z)\ge p_*,
\]
equations
\eqref{eq:reference-state-oscillation-bound},
\eqref{eq:coordinate-oscillation-as-range},
and \eqref{eq:rotated-endpoint-cost} give
\[
    \bigl(\operatorname{osc}_i H(u)\bigr)^2
    \le
    \frac{2}{p_*}
    \sum_{z\in\Omega}
    p(u_{i-1},z)
    \left[
        F\bigl(N(u)-e_{u_i}+e_z\bigr)
        -
        F\bigl(N(u)\bigr)
    \right]^2
    =
    \frac{2}{p_*}\,D_i(u).
\]
This proves
\eqref{eq:pointwise-oscillation-cost-bound}.

\medskip
\noindent
\textbf{Step 2: Density comparison under cyclic rotation.}

We next compare the \(\mu_n\)-weights of \(u\) and \(\theta_i u\).
Using
\[
    \mu_n(u)
    =
    \pi(u_1)
    \prod_{r=1}^{n-1}p(u_r,u_{r+1}),
\]
and cancelling the common transition factors, we obtain
\begin{equation}
    \frac{\mu_n(\theta_i u)}{\mu_n(u)}
    =
    \frac{\pi(u_i)}{\pi(u_1)}
    \frac{p(u_n,u_1)}{p(u_{i-1},u_i)}.
    \label{eq:rotation-density-ratio}
\end{equation}
For \(i=1\), the right-hand side equals one under the convention
\(u_0=u_n\).

Stationarity and strict positivity imply that, for every \(x\in\Omega\),
\[
    \pi(x)
    =
    \sum_{y\in\Omega}\pi(y)p(y,x)
    \ge
    p_*\sum_{y\in\Omega}\pi(y)
    =
    p_*.
\]
Together with \(\pi(x)\le1\), \(p(x,y)\le1\), and
\(p(x,y)\ge p_*\), equation
\eqref{eq:rotation-density-ratio} gives the two-sided comparison
\begin{equation}
    p_*^2
    \le
    \frac{\mu_n(\theta_i u)}{\mu_n(u)}
    \le
    \frac{1}{p_*^2}.
    \label{eq:rotation-density-comparison}
\end{equation}
In particular, since \(\theta_i\) is a bijection of \(\Omega^n\),
\begin{align}
    \mathbb{E}_{\mu_n}[D_i]
    &=
    \sum_{u\in\Omega^n}
    \mu_n(u)D_1(\theta_i u)
    \notag\\
    &=
    \sum_{v\in\Omega^n}
    \mu_n(\theta_i^{-1}v)D_1(v)
    \notag\\
    &\le
    \frac{1}{p_*^2}
    \sum_{v\in\Omega^n}
    \mu_n(v)D_1(v)
    =
    \frac{1}{p_*^2}
    \mathbb{E}_{\mu_n}[D_1].
    \label{eq:rotated-cost-expectation-bound}
\end{align}

\medskip
\noindent
\textbf{Step 3: Summation and identification of the Dirichlet form.}

Taking expectation in
\eqref{eq:pointwise-oscillation-cost-bound}, summing over
\(1\le i\le n\), and applying
\eqref{eq:rotated-cost-expectation-bound}, we obtain
\begin{align*}
    \sum_{i=1}^n
    \mathbb{E}_{\mu_n}
    \left[
        \bigl(\operatorname{osc}_i H\bigr)^2
    \right]
    &\le
    \frac{2}{p_*}
    \sum_{i=1}^n
    \mathbb{E}_{\mu_n}[D_i]
    \\
    &\le
    \frac{2n}{p_*^3}
    \mathbb{E}_{\mu_n}[D_1].
\end{align*}
Finally, using
\eqref{eq:endpoint-cost-dirichlet-form},
\[
    \sum_{i=1}^n
    \mathbb{E}_{\mu_n}
    \left[
        \bigl(\operatorname{osc}_i H\bigr)^2
    \right]
    \le
    \frac{4n}{p_*^3}
    \mathcal{E}_{\widetilde P_n}(F,F).
\]
This is exactly
\eqref{eq:coordinate-oscillation-dirichlet-bound}.
\end{proof}

\section{Upper bounds for the spectral gap}
\label{sec:upper-bounds}

The upper bound follows from the variational definition of the spectral gap.
For an eigenvalue of \(P\), we construct a function on the count space and
compute its variance and Dirichlet energy exactly.

\begin{theorem}[An explicit upper bound]
\label{thm:explicit-upper-bound}
Let \(\lambda\in(-1,1)\) be an eigenvalue of \(P\). Then, for every
\(n\ge 2\),
\begin{equation}
\operatorname{Gap}(\widetilde P_n)
\le
\frac{1-\lambda^n}
{n\frac{1+\lambda}{1-\lambda}
 -\frac{2\lambda(1-\lambda^n)}{(1-\lambda)^2}}.
\label{eq:explicit-upper-bound}
\end{equation}
\end{theorem}

\begin{proof}
Choose a nonzero eigenfunction \(f:\Omega\to\mathbb R\) satisfying $Pf=\lambda f$. 
Since \(\pi\) is stationary for \(P\), $\pi(f)=\pi(Pf)=\lambda\pi(f)$.
As \(\lambda\ne 1\), it follows that \(\pi(f)=0\). After rescaling
\(f\), we may further assume that $\pi(f^2)=1$.

Define \(F:\Omega_{n,+}^{\#}\to\mathbb R\) by
\[
F(\eta):=\sum_{x\in\Omega}\eta_x f(x).
\]
Let \(U=(U_1,\ldots,U_n)\) have law \(\mu_n\). By the definition of
the count map,
\[
F(N(U))=\sum_{i=1}^n f(U_i).
\]
Since \(\widetilde\mu_n\) is the push-forward of \(\mu_n\) under \(N\)
and each \(U_i\) has law \(\pi\),
\[
\widetilde\mu_n(F)
=
\mathbb E_{\mu_n}\left[\sum_{i=1}^n f(U_i)\right]
=
n\pi(f)
=
0.
\]

For \(1\le i<j\le n\), the Markov property, stationarity, and
\(Pf=\lambda f\) give
\begin{align*}
\mathbb E_{\mu_n}[f(U_i)f(U_j)]
=
\mathbb E_{\mu_n}
 \left[f(U_i)(P^{j-i}f)(U_i)\right]
=
\langle f,P^{j-i}f\rangle_\pi
=
\lambda^{j-i}.
\end{align*}
Moreover,
\[
\mathbb E_{\mu_n}[f(U_i)^2]=\pi(f^2)=1.
\]
Therefore
\begin{align}
\operatorname{Var}_{\widetilde\mu_n}(F)
&=
\mathbb E_{\mu_n}
\left[
\left(\sum_{i=1}^n f(U_i)\right)^2
\right]
\notag\\
&=
n+2\sum_{r=1}^{n-1}(n-r)\lambda^r
\notag\\
&=
n\frac{1+\lambda}{1-\lambda}
-\frac{2\lambda(1-\lambda^n)}{(1-\lambda)^2}.
\label{eq:exact-upper-variance}
\end{align}

We next compute the Dirichlet energy. For every \(u\in\Omega^n\)
and \(z\in\Omega\) contributing to the endpoint representation,
\[
F\bigl(N(u)-e_{u_1}+e_z\bigr)-F(N(u))
=
f(z)-f(u_1).
\]
Hence, after including the zero terms,
\begin{equation}
\mathcal E_{\widetilde P_n}(F,F)
=
\frac12
\sum_{u\in\Omega^n}\mu_n(u)
\sum_{z\in\Omega}p(u_n,z)
\bigl(f(z)-f(u_1)\bigr)^2.
\label{eq:upper-energy-sum}
\end{equation}
In the sum in \eqref{eq:upper-energy-sum}, both \(u_1\) and the
appended state \(z\) have marginal distribution \(\pi\). Moreover,
\[
\sum_{u\in\Omega^n}\mu_n(u)
\sum_{z\in\Omega}p(u_n,z)f(u_1)f(z)
=
\langle f,P^n f\rangle_\pi
=
\lambda^n.
\]
Using \(\pi(f^2)=1\) in \eqref{eq:upper-energy-sum}, we obtain
\begin{equation}
\mathcal E_{\widetilde P_n}(F,F)
=
\frac12\bigl(1+1-2\lambda^n\bigr)
=
1-\lambda^n.
\label{eq:exact-upper-energy}
\end{equation}

Since \(\lambda\in(-1,1)\), the quantity in
\eqref{eq:exact-upper-energy} is positive. Thus \(F\) is nonconstant,
and in particular its variance in \eqref{eq:exact-upper-variance} is
positive. Applying the variational definition of the spectral gap to
\(F\), and substituting \eqref{eq:exact-upper-variance} and
\eqref{eq:exact-upper-energy}, proves
\eqref{eq:explicit-upper-bound}.
\end{proof}

Theorem~\ref{thm:explicit-upper-bound} is a finite-\(n\) estimate,
valid for every \(n\ge 2\). To make its dependence on the window
length more transparent, fix \(\lambda\in(-1,1)\) and let
\(n\to\infty\). Since \(\lambda^n\to 0\), the upper bound in
\eqref{eq:explicit-upper-bound} has the asymptotic form
\[
\operatorname{Gap}(\widetilde P_n)
\le
\left(
\frac{1-\lambda}{1+\lambda}+o(1)
\right)\frac{1}{n}.
\]
This gives the leading dependence of the explicit bound on the
window length. We next derive from the same bound a simpler estimate,
expressed in terms of \(\operatorname{Gap}(P)\), that holds for every
\(n\ge 2\).

\begin{corollary}[A simpler upper bound]
\label{cor:simplified-upper-bound}
Assume that \(\operatorname{Gap}(P)<2\). Then, for every \(n\ge 2\),
\begin{equation}
\operatorname{Gap}(\widetilde P_n)
\le
\begin{cases}
\displaystyle
\frac{\operatorname{Gap}(P)}{n},
&
\operatorname{Gap}(P)\le 1,
\\[3mm]
\displaystyle
\frac{\operatorname{Gap}(P)}
{2-\operatorname{Gap}(P)}
\frac{1}{n},
&
1<\operatorname{Gap}(P)<2.
\end{cases}
\label{eq:simplified-upper-bound}
\end{equation}
\end{corollary}

\begin{proof}
Set
\[
\lambda_2:=\lambda_2(P)=1-\operatorname{Gap}(P).
\]
The assumption \(\operatorname{Gap}(P)<2\) gives
\(\lambda_2\in(-1,1)\), so Theorem~\ref{thm:explicit-upper-bound}
applies with \(\lambda=\lambda_2\).

Suppose first that \(0\le \lambda_2<1\). Since
\[
1-\lambda_2^n
=
(1-\lambda_2)\sum_{r=0}^{n-1}\lambda_2^r,
\]
Theorem~\ref{thm:explicit-upper-bound} and
\eqref{eq:exact-upper-variance} show that it suffices to prove
\begin{equation}
n+2\sum_{r=1}^{n-1}(n-r)\lambda_2^r
\ge
n\sum_{r=0}^{n-1}\lambda_2^r.
\label{eq:nonnegative-eigenvalue-denominator-bound}
\end{equation}
To verify this inequality, pair the terms with indices \(r\) and
\(n-r\). We obtain
\begin{align*}
n+2\sum_{r=1}^{n-1}(n-r)\lambda_2^r
-n\sum_{r=0}^{n-1}\lambda_2^r
&=\sum_{r=1}^{n-1}(n-2r)\lambda_2^r
\\
&
=
\sum_{r=1}^{\lfloor (n-1)/2\rfloor}
(n-2r)
\bigl(\lambda_2^r-\lambda_2^{n-r}\bigr)
\ge 0.
\end{align*}
Indeed, for every \(1\le r<n/2\), we have
\(n-2r>0\) and
\(\lambda_2^r\ge\lambda_2^{n-r}\).
This proves \eqref{eq:nonnegative-eigenvalue-denominator-bound}.
Consequently,
\[
\operatorname{Gap}(\widetilde P_n)
\le
\frac{1-\lambda_2}{n}
=
\frac{\operatorname{Gap}(P)}{n}.
\]

Now suppose that \(-1<\lambda_2<0\). By
Theorem~\ref{thm:explicit-upper-bound} and
\eqref{eq:exact-upper-variance}, it suffices to prove
\[
n\frac{1+\lambda_2}{1-\lambda_2}
-\frac{2\lambda_2(1-\lambda_2^n)}{(1-\lambda_2)^2}
\ge
n\frac{1+\lambda_2}{1-\lambda_2}
(1-\lambda_2^n).
\]
The difference between the left- and right-hand sides is
\[
\frac{-\lambda_2}{(1-\lambda_2)^2}
\left[
2(1-\lambda_2^n)
-
n(1-\lambda_2^2)\lambda_2^{n-1}
\right].
\]
If \(n\) is even, then \(\lambda_2^{n-1}<0\), so the expression
in brackets is positive.

If \(n\) is odd, set \(a:=-\lambda_2\in(0,1)\). Since
\[
n a^{n-1}
\le
\sum_{r=0}^{n-1}a^r,
\]
we have
\[
n(1-a^2)a^{n-1}
\le
(1+a)(1-a^n)
\le
2(1+a^n).
\]
Equivalently,
\[
n(1-\lambda_2^2)\lambda_2^{n-1}
\le
2(1-\lambda_2^n).
\]
Thus the required denominator bound holds for both even and odd
\(n\). Consequently,
\begin{align*}
\operatorname{Gap}(\widetilde P_n)
&\le
\frac{1-\lambda_2^n}
{n\frac{1+\lambda_2}{1-\lambda_2}
 (1-\lambda_2^n)}
\\
&=
\frac{1-\lambda_2}{1+\lambda_2}\frac{1}{n}
\\
&=
\frac{\operatorname{Gap}(P)}
{2-\operatorname{Gap}(P)}
\frac{1}{n}.
\end{align*}

\end{proof}

\begin{remark}[A uniform form when \(m\ge 3\)]
The preceding corollary can be stated without cases when \(m\ge 3\),
at the cost of a universal constant. Since \(P\) is reversible, its
eigenvalues are real, and
\[
0
\le
\operatorname{tr}(P)
=
1+\sum_{j=2}^m\lambda_j(P)
\le
1+(m-1)\lambda_2(P).
\]
It follows that
\[
\lambda_2(P)\ge -\frac{1}{m-1},
\qquad
\operatorname{Gap}(P)
\le
\frac{m}{m-1}
\le
\frac32.
\]
In particular,
\[
2-\operatorname{Gap}(P)\ge\frac12.
\]
Combining the two cases of
Corollary~\ref{cor:simplified-upper-bound}, we obtain
\[
\operatorname{Gap}(\widetilde P_n)
\le
2\,\frac{\operatorname{Gap}(P)}{n},
\qquad n\ge 2.
\]
\end{remark}

\section{The i.i.d. case: a multivariate Ehrenfest chain}
\label{sec:iid-ehrenfest}

We now consider an exactly solvable case in which the complication
caused by the hidden path ordering disappears. In particular, suppose that
\begin{equation}
    p(x,y)=\pi(y),
    \qquad x,y\in\Omega.
    \label{eq:iid-row-constant-kernel}
\end{equation}
Then the stationary chain \(X_1,X_2,\ldots\) is an i.i.d. sequence
with common law \(\pi\), and
\[
    \mu_n=\pi^{\otimes n}.
\]

For \(u=(u_1,\ldots,u_n)\) with \(N(u)=\eta\), we have
\[
    \mu_n(u)
    =
    \prod_{i=1}^n\pi(u_i)
    =
    \prod_{x\in\Omega}\pi(x)^{\eta_x}.
\]
In contrast to the general Markov case, the probability of an ordered
path now depends only on its count vector. Indeed, whenever
\(N(u)=N(v)\), we have \(\mu_n(u)=\mu_n(v)\). Consequently,
conditional on \(N(U)=\eta\), all orderings compatible with \(\eta\)
are equally likely. By symmetry,
\[
    \mathbb P(U_1=x\mid N(U)=\eta)=\frac{\eta_x}{n}.
\]
The stationary count distribution is the multinomial law
\[
    \widetilde\mu_n(\eta)
    =
    \frac{n!}{\prod_{x\in\Omega}\eta_x!}
    \prod_{x\in\Omega}\pi(x)^{\eta_x}.
\]

The count dynamics now has a simple urn interpretation. Identify each
state \(x\in\Omega\) with an urn, so that \(\eta_x\) is the number of
balls in urn \(x\). At each step, choose one of the \(n\) balls
uniformly and redistribute it independently according to \(\pi\).
Equivalently,
\begin{equation}
    \widetilde P_n(\eta,\xi)
    =
    \sum_{x,z\in\Omega}
    \frac{\eta_x}{n}\pi(z)
    \mathbf 1_{\{\xi=\eta-e_x+e_z\}}.
    \label{eq:iid-count-kernel}
\end{equation}
When \(\pi\) is uniform, the selected ball is redistributed uniformly
among the urns. The kernel in \eqref{eq:iid-count-kernel} is precisely the \(s=1\) case of the generalized
Ehrenfest urn model studied by Khare and Zhou
\cite[Section~4.3]{Khare_2009}. They determine its full spectrum in
\cite[Theorem~4.24]{Khare_2009}; with their population-size
parameter \(N\) equal to our \(n\) and \(s=1\), the eigenvalues are
\(1-k/n\), \(0\le k\le n\). Their result therefore already implies the
proposition below. We include a short self-contained proof adapted to
the Dirichlet-form framework of this paper.

\begin{proposition}[Exact spectral gap in the i.i.d. case]
\label{prop:iid-exact-gap}
Assume \eqref{eq:iid-row-constant-kernel}, with
\(\pi(x)>0\) for every \(x\in\Omega\) and \(|\Omega|\ge 2\). Then,
for every \(n\ge 2\), we have
\[
    \operatorname{Var}_{\widetilde\mu_n}(F)
    \le
    n\,\mathcal E_{\widetilde P_n}(F,F)
    \]
    and
\begin{equation*}
    \operatorname{Gap}(\widetilde P_n)=\frac1n.
\end{equation*}
\end{proposition}

\begin{proof}
We first prove the lower bound. Let
\[
    U=(U_1,\ldots,U_n)\sim\pi^{\otimes n},
\]
and let \(U'_1,\ldots,U'_n\) be independent random variables with law
\(\pi\), independent of \(U\). For \(1\le i\le n\), define
\[
    U^{(i)}
    :=
    (U_1,\ldots,U_{i-1},U'_i,U_{i+1},\ldots,U_n).
\]
Fix \(F:\Omega_n^\#\to\mathbb R\), set
\[
    H:=F\circ N,
\]
and write \(U_{-i}:=(U_j)_{j\ne i}\).

The Efron--Stein inequality in its independent-coordinate
resampling form
\cite[Theorem~3.1, p.~54]{BoucheronLugosiMassart2013}
gives
\begin{equation}
    \operatorname{Var}_{\pi^{\otimes n}}(H)
    \le
    \frac12\sum_{i=1}^n
    \mathbb E\left[
        \bigl(H(U^{(i)})-H(U)\bigr)^2
    \right].
    \label{eq:iid-efron-stein}
\end{equation}

We next identify the expression on the right-hand side of
\eqref{eq:iid-efron-stein} with the Dirichlet form of the count chain.
Using \eqref{eq:iid-count-kernel}, we obtain
\begin{align}
    \mathcal E_{\widetilde P_n}(F,F)
    &=
    \frac12
    \sum_{\eta\in\Omega_n^\#}
    \widetilde\mu_n(\eta)
    \sum_{x,z\in\Omega}
    \frac{\eta_x}{n}\pi(z)
    \bigl(
        F(\eta-e_x+e_z)-F(\eta)
    \bigr)^2
    \notag\\
    &=
    \frac{1}{2n}
    \sum_{u\in\Omega^n}
    \pi^{\otimes n}(u)
    \sum_{i=1}^n
    \sum_{z\in\Omega}\pi(z)
    \bigl(
        F(N(u)-e_{u_i}+e_z)-F(N(u))
    \bigr)^2
    \notag\\
    &=
    \frac{1}{2n}
    \sum_{i=1}^n
    \mathbb E\left[
        \bigl(H(U^{(i)})-H(U)\bigr)^2
    \right].
    \label{eq:iid-dirichlet-resampling}
\end{align}
For the second equality, we used
\[
    \sum_{i=1}^n\mathbf 1_{\{u_i=x\}}
    =
    N_x(u).
\]

Since \(\widetilde\mu_n\) is the push-forward of
\(\pi^{\otimes n}\) under \(N\), equations
\eqref{eq:iid-efron-stein} and
\eqref{eq:iid-dirichlet-resampling} imply
\[
    \operatorname{Var}_{\widetilde\mu_n}(F)
    =
    \operatorname{Var}_{\pi^{\otimes n}}(H)
    \le
    n\,\mathcal E_{\widetilde P_n}(F,F).
\]
Consequently,
\[
    \operatorname{Gap}(\widetilde P_n)\ge\frac1n.
\]

For the matching upper bound, since all rows of \(P\) are equal to
\(\pi\), we have
\[
    (Pf)(x)
    =
    \sum_{y\in\Omega}\pi(y)f(y)
    =
    \pi(f),
    \qquad x\in\Omega,
\]
for every \(f:\Omega\to\mathbb R\). Hence the eigenvalues of \(P\)
are
\[
    1,0,\ldots,0,
\]
and therefore
\[
    \operatorname{Gap}(P)=1.
\]
The first case of Corollary~\ref{cor:simplified-upper-bound} now gives
\[
    \operatorname{Gap}(\widetilde P_n)
    \le
    \frac{\operatorname{Gap}(P)}{n}
    =
    \frac1n.
\]
Combining the two bounds proves the result.
\end{proof}

\begin{remark}[An explicit eigenfunction]
For any $a\in\Omega$, define
\[
    F_a(\eta):=\eta_a-n\pi(a).
\]
By \eqref{eq:iid-count-kernel},
\[
    (\widetilde P_n F_a)(\eta)
    =
    \eta_a-\frac{\eta_a}{n}+\pi(a)-n\pi(a)
    =
    \left(1-\frac1n\right)F_a(\eta).
\]
Thus $F_a$ is a nonconstant eigenfunction of $\widetilde P_n$ with
eigenvalue $1-1/n$.  This gives a direct alternative proof of the upper
bound $\operatorname{Gap}(\widetilde P_n)\leq 1/n$.
\end{remark}

Finally, let
\[
    \pi_*:=\min_{x\in\Omega}\pi(x).
\]
In the row-constant case,
\[
    p_*=\pi_*
    \qquad\text{and}\qquad
    \delta(P)=0.
\]
Hence Theorem~\ref{thm:main-gap-lower-bound} gives
\[
    \operatorname{Gap}(\widetilde P_n)
    \ge
    \frac{\pi_*^4}{4n},
\]
whereas Proposition~\ref{prop:iid-exact-gap} gives the exact value \(1/n\).
Thus, the general lower bound has the correct \(n^{-1}\) scaling, but
its dependence on \(\pi_*\) is not sharp: although the displayed bound
deteriorates as \(\pi_*\downarrow0\), the exact spectral gap remains
\(1/n\) for every strictly positive \(\pi\).

\section{The two-state case}\label{sec:two-state}

We now consider another special setting. That is,
we  specialize in
\[
\Omega=\{0,1\}
\]
and assume throughout this section that
\[
p(x,y)>0,\qquad x,y\in\Omega.
\]
For a two-state kernel,
\begin{equation}\label{eq:two-state-base-gap}
\operatorname{Gap}(P)=p(0,1)+p(1,0).
\end{equation}
The stationary distribution is
\begin{equation}\label{eq:two-state-pi}
\pi(0)=\frac{p(1,0)}{\operatorname{Gap}(P)},
\qquad
\pi(1)=\frac{p(0,1)}{\operatorname{Gap}(P)}.
\end{equation}

A count vector \(\eta\in\Omega_n^{\#}\) is determined by the number of ones it contains. We therefore identify
\[
\eta=(n-k,k)
\quad\text{with}\quad
k\in\{0,1,\ldots,n\},
\]
and use the abbreviations
\[
\widetilde \mu_n(k):=\widetilde \mu_n(n-k,k),
\qquad
\widetilde P_n(k,\ell):=\widetilde P_n\bigl((n-k,k),(n-\ell,\ell)\bigr).
\]
Under this identification, \(\widetilde P_n\) is a reversible birth--death kernel: one transition changes the number of ones by at most one.

Define
\begin{equation}\label{eq:two-state-alpha}
\alpha_P
:=
\min_{x,x',y\in\Omega}\frac{p(x,y)}{p(x',y)}
\in(0,1].
\end{equation}
The two-state geometry yields a comparison that retains the spectral gap of the underlying chain.

\begin{theorem}[Two-state spectral-gap comparison]\label{thm:two-state-gap}
Let \(\Omega=\{0,1\}\), and assume that \(p(x,y)>0\) for all \(x,y\in\Omega\). Then, for every \(n\ge2\),
\begin{equation}\label{eq:two-state-gap}
\operatorname{Gap}(\widetilde P_n)
\ge
\alpha_P^4\,\frac{\operatorname{Gap}(P)}{n}.
\end{equation}
\end{theorem}

The factor \(\alpha_P^4\) is not uniform over all positive two-state kernels; it might deteriorate as \(P\) approaches the boundary of  positive stochastic matrices. Thus, Theorem~\ref{thm:two-state-gap} does not settle for an absolute-constant comparison. On the other hand, if the two rows of \(P\) are identical, then \(\alpha_P=1\) and \(\operatorname{Gap}(P)=1\). The theorem then gives
\[
\operatorname{Gap}(\widetilde P_n)\ge\frac1n,
\]
and the upper bound proved in the preceding section gives equality. This is identical with the result of the preceding section in the i.i.d. case.

The proof of Theorem~\ref{thm:two-state-gap} has two ingredients. First, a positive discrete derivative of the drift gives a spectral-gap lower bound for a birth--death kernel. Second, a weighted cycle argument shows that the relevant drift derivative is of order \(1/n\).

\subsection{A drift-curvature criterion for birth--death kernels}

We begin with the one-dimensional spectral-gap estimate used below. For a birth--death kernel
\(K=(K(i,j))_{0\le i,j\le n}\), adopt the boundary convention
\[
K(0,-1)=K(n,n+1)=0.
\]

\begin{lemma}[Drift-curvature criterion]\label{lem:two-state-birth-death-curvature}
Let \(K\) be an irreducible reversible Markov kernel on \(\{0,1,\ldots,n\}\), and assume that
\[
K(i,j)=0
\qquad\text{whenever }|i-j|\ge2.
\]
Suppose that, for some \(\beta>0\),
\begin{equation}\label{eq:two-state-birth-death-curvature}
K(k,k+1)-K(k,k-1)
-K(k+1,k+2)+K(k+1,k)
\ge \beta
\end{equation}
for every \(0\le k<n\). Then
\[
\operatorname{Gap}(K)\ge\beta.
\]
\end{lemma}

\begin{proof}
Set
\[
\overline K:=\frac{I+K}{2}.
\]
Fix \(0\le k<n\). We couple one step from \(k\) and one step from \(k+1\) under \(\overline K\) by assigning the following masses:
\[
\begin{array}{c|c}
(\overline X,\overline Y) & \text{probability} \\ \hline
(k-1,k+1) & \frac12 K(k,k-1) \\
(k+1,k+1) & \frac12 K(k,k+1) \\
(k,k) & \frac12 K(k+1,k) \\
(k,k+2) & \frac12 K(k+1,k+2),
\end{array}
\]
and placing the remaining mass on \((k,k+1)\). At the two boundaries, the corresponding zero-mass outcomes in the table are omitted. The remaining mass is nonnegative because the total off-diagonal mass in each row of \(K\) is at most one. A direct check shows that
\[
\overline X\sim\overline K(k,\cdot),
\qquad
\overline Y\sim\overline K(k+1,\cdot).
\]
Moreover, by \eqref{eq:two-state-birth-death-curvature},
\begin{align*}
\mathbb E|\overline X-\overline Y|
&=1+\frac12 K(k,k-1)+\frac12 K(k+1,k+2)
   -\frac12 K(k,k+1)-\frac12 K(k+1,k)\\
&\le 1-\frac\beta2.
\end{align*}

For \(f:\{0,1,\ldots,n\}\to\mathbb R\), let
\[
\operatorname{Lip}(f):=
\max_{0\le k<n}|f(k+1)-f(k)|.
\]
The preceding coupling gives, for every \(0\le k<n\),
\[
|\overline Kf(k+1)-\overline Kf(k)|
\le
\operatorname{Lip}(f)\,\mathbb E|\overline X-\overline Y|
\le
\left(1-\frac\beta2\right)\operatorname{Lip}(f).
\]
Hence
\begin{equation}\label{eq:two-state-Lipschitz-contraction}
\operatorname{Lip}(\overline Kf)
\le
\left(1-\frac\beta2\right)\operatorname{Lip}(f).
\end{equation}

Let \(f\) be a nonconstant eigenfunction of \(K\) with eigenvalue \(\lambda_2(K)\). Then
\[
\overline Kf=\frac{1+\lambda_2(K)}2 f.
\]
Since \(\lambda_2(K)\ge-1\) and \(\operatorname{Lip}(f)>0\), \eqref{eq:two-state-Lipschitz-contraction} implies
\[
\frac{1+\lambda_2(K)}2
\le
1-\frac\beta2.
\]
Therefore \(1-\lambda_2(K)\ge\beta\), as claimed.
\end{proof}

\subsection{A weighted cycle estimate}

Under the identification of the count state \((n-k,k)\) with \(k\), we also write
\[
u\sim k
\qquad\text{when}\qquad
N_1(u)=k.
\]
For \(U\sim\mu_n\), define
\begin{equation}\label{eq:two-state-r-k}
r_k
:=
\mathbb P\bigl(U_1=0\mid N_1(U)=k\bigr).
\end{equation}
The next proposition is the quantitative combinatorial input in the proof of Theorem~\ref{thm:two-state-gap}.

\begin{proposition}[Endpoint monotonicity]\label{prop:two-state-endpoint-monotonicity}
For every \(0\le k<n\),
\begin{equation}\label{eq:two-state-endpoint-monotonicity}
\frac{\alpha_P^4}{n}
\le
r_k-r_{k+1}
\le
\frac{\alpha_P^{-4}}{n}.
\end{equation}
\end{proposition}

\begin{proof}
Fix \(0\le k<n\). Define the set of admissible pairs
\begin{align}\label{eq:two-state-admissible-pairs}
\mathcal A_k:=\biggl\{(w,v)\in\Omega^n\times\Omega^n:\;&
 w\sim k,\ v\sim k+1,\ w_1=0,\ v_1=1,\\[-2mm]
&\sum_{i=2}^t w_i\ge\sum_{i=2}^t v_i
\quad\text{for every }2\le t\le n
\biggr\}.\nonumber
\end{align}
We first show that
\begin{equation}\label{eq:two-state-reflection-identity}
(r_k-r_{k+1})\widetilde\mu_n(k)\widetilde\mu_n(k+1)
=
\sum_{(w,v)\in\mathcal A_k}\mu_n(w)\mu_n(v).
\end{equation}

By the definition of $r_k$,
\begin{align}
(r_k-r_{k+1}) \widetilde\mu_n(k)\widetilde\mu_n(k+1) 
 = \left(\sum_{\substack{w\sim k \\ w_1=0}}\mu_n(w)\right) \widetilde\mu_n(k+1)
- \left(\sum_{\substack{w\sim k+1 \\ w_1=0}}\mu_n(w)\right) \widetilde\mu_n(k).
\label{eq:two-state-cross-product-first}
\end{align}
For each $\ell$, decompose the total mass
\[
\widetilde\mu_n(\ell)
=
\sum_{\substack{u\sim \ell \\ u_1=0}}\mu_n(u)
+
\sum_{\substack{u\sim \ell \\ u_1=1}}\mu_n(u).
\]
Substituting these decompositions into
\eqref{eq:two-state-cross-product-first}, the terms in which both paths start from $0$ cancel. We therefore obtain
\begin{align}\label{eq:two-state-cross-product}
(r_k-r_{k+1}) \widetilde\mu_n(k)\widetilde\mu_n(k+1)
 = \sum_{\substack{w\sim k,\; v\sim k+1 \\ w_1=0,\; v_1=1}}
\mu_n(w)\mu_n(v)
- \sum_{\substack{w\sim k+1,\; v\sim k \\ w_1=0,\; v_1=1}}
\mu_n(w)\mu_n(v).
\end{align}

We now partition the pairs in the first sum according to whether they belong to $\mathcal A_k$. Their admissible contribution is
\[
\sum_{(w,v)\in\mathcal A_k}\mu_n(w)\mu_n(v),
\]
which is the right-hand side of
\eqref{eq:two-state-reflection-identity}. It remains to prove that the total weight of the non-admissible pairs in the first sum is exactly the second sum in
\eqref{eq:two-state-cross-product}. We establish this equality by constructing a weight-preserving bijection between the two corresponding classes of pairs.

Fix \(w\sim k\) and \(v\sim k+1\) such that
\[
w_1=0,
\qquad
v_1=1,
\qquad
(w,v)\notin\mathcal A_k.
\]
Since
\[
\sum_{i=2}^n(w_i-v_i)=0,
\]
the failure of the inequalities in \eqref{eq:two-state-admissible-pairs} is equivalent to the existence of an index \(t\) such that
\[
\sum_{i=t+1}^n(w_i-v_i)>0.
\]
Let \(j\) be the largest such index. Maximality gives
\begin{equation}\label{eq:two-state-last-positive-suffix}
w_{j+1}=1,
\qquad
v_{j+1}=0,
\qquad
\sum_{i=j+2}^n w_i=\sum_{i=j+2}^n v_i.
\end{equation}
Define
\begin{align*}
\widehat w
&:=(w_1,v_j,v_{j-1},\ldots,v_2,w_{j+1},\ldots,w_n),\\
\widehat v
&:=(v_1,w_j,w_{j-1},\ldots,w_2,v_{j+1},\ldots,v_n).
\end{align*}
We first verify that the image pair belongs to the class indexed by the second sum in \eqref{eq:two-state-cross-product}. Using \eqref{eq:two-state-last-positive-suffix} and
\[
\sum_{i=1}^n(w_i-v_i)=-1,
\qquad
w_1-v_1=-1,
\]
we obtain
\[
\sum_{i=2}^j(w_i-v_i)=-1.
\]
Consequently,
\[
\widehat w\sim k+1,
\qquad
\widehat v\sim k,
\qquad
\widehat w_1=0,
\qquad
\widehat v_1=1.
\]

We next verify that the transformation preserves the weight assigned to each pair:
\begin{equation}\label{eq:two-state-reflection-weight}
\mu_n(w)\mu_n(v)
=
\mu_n(\widehat w)\mu_n(\widehat v).
\end{equation}
Indeed, write
\[
T
:=
\prod_{i=j+1}^{n-1}
p(w_i,w_{i+1})p(v_i,v_{i+1}),
\]
with the convention that $T=1$ when $j=n-1$. Since
\[
(w_1,v_1)=(0,1),
\qquad
(w_{j+1},v_{j+1})=(1,0),
\]
reversibility of $P$ gives
\begin{align*}
\mu_n(w)\mu_n(v)
&=
\pi(w_1)\pi(v_1)
\prod_{i=1}^{j}
p(w_i,w_{i+1})p(v_i,v_{i+1}) \cdot T
\\
&=
\pi(w_{j+1})\pi(v_{j+1})
\prod_{i=1}^{j}
p(w_{i+1},w_i)p(v_{i+1},v_i) \cdot T.
\end{align*}
By the definition of $\widehat w$ and $\widehat v$,
\[
(\widehat w_1,\ldots,\widehat w_{j+1})
=
(v_{j+1},v_j,\ldots,v_1),
\]
and
\[
(\widehat v_1,\ldots,\widehat v_{j+1})
=
(w_{j+1},w_j,\ldots,w_1).
\]
Moreover, the two paths are unchanged after coordinate $j+1$. Therefore,
\begin{align*}
\mu_n(w)\mu_n(v)
&=
\pi(\widehat w_1)\pi(\widehat v_1)
\prod_{i=1}^{j}
p(\widehat w_i,\widehat w_{i+1})
p(\widehat v_i,\widehat v_{i+1}) \cdot T
\\
&=
\mu_n(\widehat w)\mu_n(\widehat v),
\end{align*}
which proves \eqref{eq:two-state-reflection-weight}.

It remains to show that the preceding transformation is a bijection from the non-admissible pairs in the first sum of \eqref{eq:two-state-cross-product} onto the pairs in the second sum.

We first prove injectivity. It is enough to show that the cutting index $j$ can be recovered uniquely from the image pair $(\widehat w,\widehat v)$. By construction, the coordinates after $j+1$ are unchanged:
\[
(\widehat w_i,\widehat v_i)=(w_i,v_i), \qquad i\ge j+1.
\]
Hence $j$ is still the largest index $t$ such that
\[
\sum_{i=t+1}^n \widehat w_i > \sum_{i=t+1}^n \widehat v_i.
\]
Thus $j$ is determined by $(\widehat w,\widehat v)$. Once $j$ is known, reversing the first $j+1$ coordinates and exchanging the two paths recovers $(w,v)$. Therefore the transformation is injective.

We next prove surjectivity. Let $(\widehat w,\widehat v)$ be an arbitrary pair appearing in the second sum of \eqref{eq:two-state-cross-product}; that is,
\[
\widehat w\sim k+1, \qquad \widehat v\sim k, \qquad \widehat w_1=0, \qquad \widehat v_1=1.
\]
Since the two paths contain $k+1$ and $k$ ones, respectively, while their first coordinates are $0$ and $1$, we have
\[
\sum_{i=2}^n \widehat w_i > \sum_{i=2}^n \widehat v_i.
\]
Consequently, there exists at least one index $t\in\{1,\ldots,n-1\}$ such that
\[
\sum_{i=t+1}^n \widehat w_i > \sum_{i=t+1}^n \widehat v_i.
\]
Let $j$ be the largest such index, and define
\[
w = (\widehat w_1,\widehat v_j,\ldots,\widehat v_2, \widehat w_{j+1},\ldots,\widehat w_n),
\]
and
\[
v = (\widehat v_1,\widehat w_j,\ldots,\widehat w_2, \widehat v_{j+1},\ldots,\widehat v_n).
\]
This is precisely the inverse reversal-and-exchange operation. The choice of $j$ implies that
\[
w\sim k, \qquad v\sim k+1, \qquad w_1=0, \qquad v_1=1,
\]
and that $(w,v)\notin\mathcal A_k$. Applying the original transformation to $(w,v)$ returns $(\widehat w,\widehat v)$. Hence every pair in the second sum has a preimage, and the transformation is surjective.

We have therefore constructed a weight-preserving bijection between the non-admissible pairs in the first sum and all pairs in the second sum. The two non-admissible contributions in \eqref{eq:two-state-cross-product} cancel, proving \eqref{eq:two-state-reflection-identity}.

It remains to estimate the weight of \(\mathcal A_k\). For \(1\le s\le n\), use the cyclic rotation
\[
\theta_s(u_1,\ldots,u_n)
:=(u_s,u_{s+1},\ldots,u_n,u_1,\ldots,u_{s-1}),
\]
which is the same rotation used in the proof of Proposition~\ref{prop:coordinate-oscillation-dirichlet-bound}. For a pair \(w\sim k\), \(v\sim k+1\), put
\[
a_i:=w_i-v_i,
\qquad
S_0:=0,
\qquad
S_t:=\sum_{i=1}^t a_i.
\]
Since \(S_n=-1\), let \(\tau\) be the smallest index at which \(S_t\) attains its minimum over \(1\le t\le n\). The increments \(a_i\) belong to \(\{-1,0,1\}\), so
\[
a_\tau=-1,
\qquad
S_{\tau-1}=S_\tau+1.
\]
Moreover,
\[
S_t\ge S_\tau\quad(t\ge\tau),
\qquad
S_t\ge S_\tau+1=S_{\tau-1}\quad(t<\tau),
\]
where the second inequality uses the choice of the first minimizer. These inequalities are exactly the prefix inequalities in \eqref{eq:two-state-admissible-pairs} for the rotated pair, and hence
\[
(\theta_\tau w,\theta_\tau v)\in\mathcal A_k.
\]
Conversely, the same prefix inequalities show that if \((\theta_s w,\theta_s v)\in\mathcal A_k\), then \(s\) is the smallest index at which \(S_t\) attains its minimum. Thus exactly one of the \(n\) simultaneously rotated pairs
\[
(\theta_s w,\theta_s v),
\qquad 1\le s\le n,
\]
belongs to \(\mathcal A_k\). This is the elementary cycle lemma in the present form.

We next compare the weights of cyclic rotations. For \(u\in\Omega^n\) and \(2\le s\le n\), cancellation in the open-path density gives
\begin{equation}\label{eq:two-state-binary-rotation-ratio}
\frac{\mu_n(\theta_su)}{\mu_n(u)}
=
\frac{\pi(u_s)p(u_n,u_1)}{\pi(u_1)p(u_{s-1},u_s)}.
\end{equation}
Using \eqref{eq:two-state-pi}, we can rewrite the right-hand side as
\[
\frac{p(u_n,u_1)}{p(1-u_1,u_1)}
\frac{p(1-u_s,u_s)}{p(u_{s-1},u_s)}.
\]
Each factor compares two transition probabilities with the same terminal state. By \eqref{eq:two-state-alpha},
\begin{equation}\label{eq:two-state-binary-rotation-comparison}
\alpha_P^2
\le
\frac{\mu_n(\theta_su)}{\mu_n(u)}
\le
\alpha_P^{-2}.
\end{equation}
The same bounds are trivial for \(s=1\). Hence, for every pair \((w,v)\),
\begin{equation}\label{eq:two-state-pair-rotation-comparison}
\alpha_P^4
\le
\frac{\mu_n(\theta_sw)\mu_n(\theta_sv)}{\mu_n(w)\mu_n(v)}
\le
\alpha_P^{-4}.
\end{equation}

The uniqueness of the admissible rotation gives the exact double-counting identity
\begin{align}\label{eq:two-state-cycle-double-count}
\widetilde\mu_n(k)\widetilde\mu_n(k+1)
=
\sum_{s=1}^n
\sum_{(w,v)\in\mathcal A_k}
\mu_n(\theta_s^{-1}w)\mu_n(\theta_s^{-1}v).
\end{align}
Applying \eqref{eq:two-state-pair-rotation-comparison} to \eqref{eq:two-state-cycle-double-count}, we obtain
\begin{equation}\label{eq:two-state-admissible-weight-bounds}
\frac{\alpha_P^4}{n}
\widetilde\mu_n(k)\widetilde\mu_n(k+1)
\le
\sum_{(w,v)\in\mathcal A_k}\mu_n(w)\mu_n(v)
\le
\frac{\alpha_P^{-4}}{n}
\widetilde\mu_n(k)\widetilde\mu_n(k+1).
\end{equation}
Combining \eqref{eq:two-state-reflection-identity} and \eqref{eq:two-state-admissible-weight-bounds} proves \eqref{eq:two-state-endpoint-monotonicity}.
\end{proof}

\subsection{Proof of the two-state comparison}

\begin{proof}[Proof of Theorem~\ref{thm:two-state-gap}]
Let \(U=(U_1,\ldots,U_n)\) have law \(\mu_n\), and, conditional on \(U\), let \(Z\) be sampled from \(P(U_n,\cdot)\). Since one transition of the count chain changes the number of ones by \(Z-U_1\), for every \(0\le k\le n\),
\begin{equation}\label{eq:two-state-drift-expectation}
\widetilde P_n(k,k+1)-\widetilde P_n(k,k-1)
=
\mathbb E[Z-U_1\mid N_1(U)=k].
\end{equation}
Here and below, we use the boundary convention
\[
\widetilde P_n(0,-1)=\widetilde P_n(n,n+1)=0.
\]

We next compute the conditional expectation on the right-hand side. Conditional on the hidden path $U$, the appended state $Z$ is sampled from $P(U_n,\cdot)$. Thus
\begin{align}
\mathbb E[Z\mid N_1(U)=k]
&=
\mathbb E\left[
\mathbb E[Z\mid U]
\;\middle|\;
N_1(U)=k
\right]
\nonumber\\
&=
\mathbb E\left[
p(U_n,1)
\;\middle|\;
N_1(U)=k
\right].
\label{eq:two-state-conditional-Z}
\end{align}
By \eqref{eq:two-state-base-gap},
\[
p(x,1)
=
p(0,1)
+
\bigl(1-\operatorname{Gap}(P)\bigr)x,
\qquad x\in\{0,1\}.
\]

Reversibility of the stationary path law implies that, conditional on $N_1(U)=k$, the two endpoints $U_1$ and $U_n$ have the same distribution. By the definition of $r_k$,
\[
\mathbb E[U_n\mid N_1(U)=k]
=
\mathbb E[U_1\mid N_1(U)=k]
=
\mathbb P(U_1=1\mid N_1(U)=k)
=
1-r_k.
\]
Substituting this identity into
\eqref{eq:two-state-conditional-Z} gives
\[
\mathbb E[Z\mid N_1(U)=k]
=
p(0,1)
+
\bigl(1-\operatorname{Gap}(P)\bigr)(1-r_k).
\]
Then \eqref{eq:two-state-drift-expectation} yields
\begin{align}
\widetilde P_n(k,k+1)-\widetilde P_n(k,k-1)
&=
p(0,1)
+
\bigl(1-\operatorname{Gap}(P)\bigr)(1-r_k)
-
(1-r_k)
\nonumber\\
&=
\operatorname{Gap}(P)r_k-p(1,0).
\label{eq:two-state-drift}
\end{align}

Consequently, for every \(0\le k<n\),
\begin{align*}
&\widetilde P_n(k,k+1)-\widetilde P_n(k,k-1)
-\widetilde P_n(k+1,k+2)+\widetilde P_n(k+1,k)\\
&\hspace{20mm}=\operatorname{Gap}(P)(r_k-r_{k+1})
\ge
\alpha_P^4\frac{\operatorname{Gap}(P)}{n},
\end{align*}
where the last inequality is Proposition~\ref{prop:two-state-endpoint-monotonicity}. Since \(\widetilde P_n\) is a reversible birth--death kernel, Lemma~\ref{lem:two-state-birth-death-curvature} gives
\[
\operatorname{Gap}(\widetilde P_n)
\ge
\alpha_P^4\frac{\operatorname{Gap}(P)}{n}.
\]
\end{proof}

\section{Matrix concentration from the count-space Poincaré inequality}
\label{sec:matrix-concentration}

In this section, we study applications of the count-space Poincar\'e
inequality for matrix-valued empirical averages.
Applying the scalar Poincaré inequality to a fixed matrix entry or a fixed quadratic form gives only scalar, directionwise control. Our purpose is to obtain simultaneous operator-norm control by combining the count-space Poincaré inequality with a general matrix-concentration principle.

Let $\mathbb{H}_d$ denote the real vector space of $d\times d$ Hermitian matrices,
equipped with the $\ell_2$ operator norm $\|\cdot\|$.  For a function
$g:\Omega\to\mathbb{H}_d$, write
\[
    \pi(g):=\sum_{x\in\Omega}\pi(x)g(x)\in\mathbb{H}_d.
\]
The next result converts the scalar Poincar\'e inequality for the resampled count
kernel into operator-norm concentration for a matrix-valued empirical average.

\begin{theorem}[Matrix concentration for a stationary path segment]
\label{thm:matrix-empirical-average}
Let $P$ be an irreducible Markov kernel on the finite state space $\Omega$ that is
reversible with respect to $\pi$, and fix $n\ge 2$.  Let
$U=(U_1,\ldots,U_n)$ have law $\mu_n$.  If
$g:\Omega\to\mathbb{H}_d$ satisfies
\[
    \max_{x\in\Omega}\|g(x)\|\le 1,
\]
then, for every $t>0$,
\begin{equation}
\label{eq:matrix-tail-gap}
\mathbb{P}_{\mu_n}\left(
    \left\|\frac{1}{n}\sum_{i=1}^n g(U_i)-\pi(g)\right\|\ge t
\right)
\le
6d\exp\left(
    -n\sqrt{\frac{\operatorname{Gap}(\widetilde P_n)}{2}}\,t
\right).
\end{equation}
Moreover,
\begin{equation}
\label{eq:matrix-mean-gap}
\mathbb{E}_{\mu_n}
\left\|\frac{1}{n}\sum_{i=1}^n g(U_i)-\pi(g)\right\|
\le
\frac{\sqrt{2}\,\log(6\mathrm{e}d)}
     {n\sqrt{\operatorname{Gap}(\widetilde P_n)}}.
\end{equation}
\end{theorem}

\begin{proof}
Consider the continuous-time Markov semigroup on $\Omega_{n,+}^{\#}$ given by
\begin{equation}
\label{eq:poissonized-count-semigroup}
    \widehat P_t:=\exp\bigl(t(\widetilde P_n-I)\bigr),
    \qquad t\ge 0.
\end{equation}
Since $\widetilde P_n$ is irreducible and reversible with respect to
$\widetilde\mu_n$, the semigroup in
\eqref{eq:poissonized-count-semigroup} is ergodic and reversible with the same
stationary measure.  Its scalar Dirichlet form is $\mathcal E_{\widetilde P_n}$.
Therefore, by the variational definition of the spectral gap, every scalar
function $h:\Omega_{n,+}^{\#}\to\mathbb R$ satisfies
\begin{equation}
\label{eq:poissonized-count-poincare}
    \operatorname{Var}_{\widetilde\mu_n}(h)
    \le
    \frac{1}{\operatorname{Gap}(\widetilde P_n)}
    \mathcal E_{\widetilde P_n}(h,h).
\end{equation}

Define the matrix-valued count observable
\begin{equation}
\label{eq:matrix-count-observable}
    \bm F_g(\eta):=\frac{1}{n}\sum_{x\in\Omega}\eta_x g(x),
    \qquad \eta\in\Omega_{n,+}^{\#}.
\end{equation}
Because $\widetilde\mu_n$ is the push-forward of $\mu_n$ under $N$ and each
coordinate of a path with law $\mu_n$ has marginal distribution $\pi$, we have
\begin{equation}
\label{eq:matrix-count-mean}
    \mathbb E_{\widetilde\mu_n}\bm F_g
    =\mathbb E_{\mu_n}\bm F_g(N(U))
    =\frac{1}{n}\sum_{i=1}^n\mathbb E_{\mu_n}g(U_i)
    =\pi(g).
\end{equation}

The matrix carr\'e du champ associated with the semigroup
\eqref{eq:poissonized-count-semigroup} is
\begin{equation}
\label{eq:matrix-carre-du-champ}
\bm\Gamma_n(\bm F_g)(\eta)
:=\frac{1}{2}\sum_{\xi\in\Omega_{n,+}^{\#}}
\widetilde P_n(\eta,\xi)
\bigl(\bm F_g(\xi)-\bm F_g(\eta)\bigr)^2.
\end{equation}
Every nontrivial transition of $\widetilde P_n$ replaces one count with another.
Thus, whenever $\widetilde P_n(\eta,\xi)>0$ and $\eta\ne\xi$, there are
$x,z\in\Omega$ such that $\xi=\eta-e_x+e_z$, and hence
\[
    \bm F_g(\xi)-\bm F_g(\eta)=\frac{1}{n}\bigl(g(z)-g(x)\bigr).
\]
Since $g(z)-g(x)$ is Hermitian and its norm is at most $2$, it follows from
\eqref{eq:matrix-carre-du-champ} that
\begin{equation} \label{eq:matrix-carre-du-champ-bound}
\left\|\bm\Gamma_n(\bm F_g)(\eta)\right\|
\le \frac{1}{2}\sum_{\xi\in\Omega_{n,+}^{\#}}
\widetilde P_n(\eta,\xi)
\left\|\bm F_g(\xi)-\bm F_g(\eta)\right\|^2 
 \le \frac{2}{n^2}.
\end{equation}

Apply Theorem~2.7 of \cite{huang2021poincare} to the semigroup
\eqref{eq:poissonized-count-semigroup}.  In view of
\eqref{eq:poissonized-count-poincare} and
\eqref{eq:matrix-carre-du-champ-bound}, its Poincar\'e constant and variance proxy
may be taken to be
\[
    \alpha=\frac{1}{\operatorname{Gap}(\widetilde P_n)},
    \qquad
    v_{\bm F_g}\le \frac{2}{n^2}.
\]
Consequently, for every $\lambda>0$,
\begin{equation}
\label{eq:huang-tropp-applied}
\mathbb P_{\widetilde\mu_n}\left(
    \left\|\bm F_g-\mathbb E_{\widetilde\mu_n}\bm F_g\right\|
    \ge
    \sqrt{\frac{2}{n^2\operatorname{Gap}(\widetilde P_n)}}\,\lambda
\right)
\le 6d\exp(-\lambda).
\end{equation}
Set
$\lambda=n\sqrt{\operatorname{Gap}(\widetilde P_n)/2}\,t$ in
\eqref{eq:huang-tropp-applied}.  Using
\eqref{eq:matrix-count-observable} and \eqref{eq:matrix-count-mean}, together
with the push-forward relation between $\mu_n$ and $\widetilde\mu_n$, gives
\eqref{eq:matrix-tail-gap}.  The expectation estimate in Theorem~2.7 of
\cite{huang2021poincare}, with the same values of $\alpha$ and
$v_{\bm F_g}$, gives \eqref{eq:matrix-mean-gap}.
\end{proof}

\begin{corollary}[Explicit bound under strict positivity]
\label{cor:matrix-empirical-average-positive}
Assume in addition that
$p_*:=\min_{x,y\in\Omega}p(x,y)>0$, and set
\begin{equation}
\label{eq:kappa-for-matrix-concentration}
    \kappa(P):=
    \frac{p_*^4\bigl(1-\sqrt{\delta(P)}\bigr)^2}{4}.
\end{equation}
Then, under the assumptions of Theorem~\ref{thm:matrix-empirical-average},
\begin{equation}
\label{eq:matrix-tail-explicit}
\mathbb{P}_{\mu_n}\left(
    \left\|\frac{1}{n}\sum_{i=1}^n g(U_i)-\pi(g)\right\|\ge t
\right)
\le
6d\exp\left(-\sqrt{\frac{\kappa(P)n}{2}}\,t\right),
\qquad t>0,
\end{equation}
and
\begin{equation}
\label{eq:matrix-mean-explicit}
\mathbb{E}_{\mu_n}
\left\|\frac{1}{n}\sum_{i=1}^n g(U_i)-\pi(g)\right\|
\le
\log(6\mathrm{e}d)\sqrt{\frac{2}{\kappa(P)n}}.
\end{equation}
\end{corollary}

\begin{proof}
Theorem~\ref{thm:main-gap-lower-bound} gives
\[
    \operatorname{Gap}(\widetilde P_n)\ge \frac{\kappa(P)}{n}.
\]
Substituting this estimate into \eqref{eq:matrix-tail-gap} and
\eqref{eq:matrix-mean-gap} proves \eqref{eq:matrix-tail-explicit} and
\eqref{eq:matrix-mean-explicit}.
\end{proof}

\begin{remark}

The semigroup in \eqref{eq:poissonized-count-semigroup} is used only as an auxiliary
reversible dynamics with stationary law $\widetilde{\mu}_n$; it is not the
multi-step dynamics of the original sliding-window count process.
The resulting bound is subexponential in the deviation parameter
\cite{huang2021poincare}. In the overlapping setting, direct concentration
inequalities for additive matrix functionals of Markov chains yield stronger
Hoeffding- or Bernstein-type tails \cite{neeman2024concentration}.
Thus, the purpose of this section is to illustrate a consequence of the
count-space Poincaré inequality, rather than to obtain an optimal matrix
concentration bound; moreover, Corollary~\ref{cor:matrix-empirical-average-positive} inherits the possibly nonsharp
dependence on $P$ from Theorem~\ref{thm:main-gap-lower-bound}.

\end{remark}

\section{The general spectral-gap comparison problem}
\label{sec:comparison-problem}

Under the strict positivity assumption, the preceding sections determine the
dependence of $\operatorname{Gap}(\widetilde P_n)$ on the window length for
each fixed underlying kernel. The remaining question is whether the dependence
on $P$ can be described universally through $\operatorname{Gap}(P)$, without
requiring all transition probabilities to be positive.


\subsection{What is proved under strict positivity}
\label{subsec:positive-kernel-results}

For every fixed strictly positive reversible kernel, the lower bound of
Theorem~\ref{thm:main-gap-lower-bound} and the upper bounds of Section~\ref{sec:upper-bounds}
determine the correct order in the window length.

\begin{corollary}[Matching order for fixed positive kernels]
\label{cor:matching-order}
Let $P$ be a reversible Markov kernel on a finite state space $\Omega$ with
$|\Omega|\geq 2$, and assume that
\[
    p_*:=\min_{x,y\in\Omega}p(x,y)>0.
\]
Then there exist constants
\[
    0<c_-(P)\leq c_+(P)<\infty
\]
such that, for every $n\geq 2$,
\begin{equation}
    \frac{c_-(P)}{n}
    \leq
    \operatorname{Gap}(\widetilde P_n)
    \leq
    \frac{c_+(P)}{n}.
    \label{eq:matching-order}
\end{equation}
In particular,
\begin{equation}
    \operatorname{Gap}(\widetilde P_n)
    =
    \Theta_P(n^{-1}).
    \label{eq:theta-fixed-P}
\end{equation}
\end{corollary}

\begin{proof}
The lower bound follows from
Theorem~\ref{thm:main-gap-lower-bound}. Strict positivity implies that $P$ is
primitive and hence $\operatorname{Gap}(P)<2$. The upper bound then follows
from Corollary~\ref{cor:simplified-upper-bound}, with the resulting
$P$-dependent quantity absorbed into $c_+(P)$.
\end{proof}

Thus the dependence on $n$ is completely determined for each fixed positive
kernel. The remaining issue is the dependence of the constants on $P$.
The lower bound in Theorem~\ref{thm:main-gap-lower-bound} is expressed through
the local parameters $p_*$ and $\delta(P)$, whereas the upper bounds in
Section~\ref{sec:upper-bounds} are naturally expressed through the Poincar\'e spectral gap
$\operatorname{Gap}(P)$. This raises the possibility that the dependence on
$p_*$ and $\delta(P)$ reflects the present comparison argument rather than
the intrinsic scale of the count kernel.

\subsection{Why the present proof requires strict positivity}
\label{subsec:method-limitations}

Strict positivity enters both main steps in the proof of
Theorem~\ref{thm:main-gap-lower-bound}.

The first step is the path-space variance estimate
\eqref{eq:path-variance-bound}. Its coupling argument uses
$\delta(P)<1$ to obtain geometric decay of disagreement probabilities.
If zero transition probabilities are allowed, one may have $\delta(P)=1$
even when $P$ is irreducible and has a positive spectral gap. The same argument compares the distribution of an intermediate path with the
corresponding conditional Markov path law. At the interface between its two
pieces, one transition factor may be missing; strict positivity yields the
required pointwise domination at a cost of $p_*^{-1}$. When $p_*=0$, the
intermediate path may contain a forbidden transition, and this domination may
fail.

The second step is the coordinate-to-endpoint comparison
\eqref{eq:coordinate-oscillation-dirichlet-bound}. The pointwise oscillation
estimate applies Lemma~\ref{lem:reference-state-oscillation-bound} with
$q(z)=p(u_{i-1},z)$, and therefore requires a uniform positive lower bound on
the entries of this transition row. The density comparison under cyclic
rotation likewise uses strict positivity to compare
$\mu_n(\theta_i u)$ with $\mu_n(u)$. When transition probabilities vanish, a
cyclic rotation of a path with positive $\mu_n$-mass may have zero mass, so
these two estimates are not available in their present form.


Therefore, it remains open to establish a universal lower bound without the strict positivity condition. The proof would  require a support-sensitive comparison controlled by the global relaxation of $P$, rather than by the smallest positive transition
probability.

\section*{Acknowledgment}
The authors are grateful to Dr. De Huang, who initiated our studies for this work.   

\bibliographystyle{plain}
\bibliography{reference}

\end{document}